\documentclass[a4paper, 12pt]{article}
\usepackage [a4paper,left=2.5cm,bottom=2.5cm,right=2.5cm,top=2.5cm]{geometry}
\usepackage[english]{babel}
\usepackage{amssymb}
\usepackage{amsmath,amsthm, dsfont}
\usepackage{subcaption}
\usepackage{comment}
\usepackage{tabularx,array}
\usepackage{graphicx, url}
\usepackage{enumitem} 
\usepackage{bm}
\usepackage{float}
\usepackage{bbm}
\usepackage{xcolor}
\usepackage{graphicx}
\usepackage{todonotes}
\usepackage{todonotes}
\usepackage{hyperref}
\hypersetup{colorlinks=true, linkcolor=blue,citecolor=gray}
\usepackage[nameinlink,capitalise]{cleveref}
\crefformat{equation}{#2(#1)#3}

\usepackage{color,soul}

\theoremstyle{plain}
\theoremstyle{plain}
\crefname{assumption}{Assumption}{Assumptions}
\newtheorem{theorem}{Theorem}[section]
\newtheorem{lemma}[theorem]{Lemma}

\newtheorem{proposition}[theorem]{Proposition}

\newtheorem{remark}[theorem]{Remark}
\newtheorem{definition}{Definition}

\newcommand{\argmin}[1]{\underset{#1}
{\operatorname{arg}\!\operatorname{min}}\;}

\newcommand{\E}{\mathbb{E}}

\newcommand{\R}{\mathbb{R}}

\renewcommand{\P}{\mathbb{P}}

\newcommand{\thetaS}{\theta_{0,\supp}}
\newcommand{\thetaSc}{\theta_{0,\supp^C}}

\newcommand{\nlip}[1]{\| #1 \|_{\text{Lip}}}

\newcommand{\bt}{b_{\theta}}

\newcommand{\bto}{b_{\theta_0}}

\newcommand{\nop}[1]{\| #1 \|_{\text{op}}}

\newcommand{\alasso}{\widehat{\theta}}

\newcommand{\sign}{\text{sign}}
\newcommand{\M}[2]{\mathcal{M}_{#1 \times #2}(\R)}

\newcommand{\supp}{\mathcal{S}}

\title{Consistent support recovery for high-dimensional diffusions}
\author{Dmytro Marushkevych \thanks{School of Mathematics and Statistics, UNSW Sydney, 4111 Anita B. Lawrence Centre, Sydney NSW 2052, Australia, 
Email: d.marushkevych@unsw.edu.au}\and  Francisco Pina \thanks{Department of Mathematics, University of Luxembourg, Maison du Nombre, 6 Avenue de la Fonte, 4364 Esch-sur-Alzette,
Luxembourg, Email: francisco.pina@uni.lu} \and 
Mark Podolskij\thanks{Department of Mathematics, University of Luxembourg, Maison du Nombre, 6 Avenue de la Fonte, 4364 Esch-sur-Alzette,
Luxembourg, Email: mark.podolskij@uni.lu}
}

\usepackage{autonum}
\numberwithin{equation}{section}

\date{\today}
\begin{document}
\maketitle

\begin{abstract}
Statistical inference for stochastic processes has advanced significantly due to applications in diverse fields, but challenges remain in high-dimensional settings where parameters are allowed to grow with the sample size. This paper analyzes a $d$-dimensional ergodic diffusion process under sparsity constraints, focusing on the adaptive Lasso estimator, which improves variable selection and bias over the standard Lasso.
We derive conditions under which the adaptive Lasso achieves support recovery property and asymptotic normality for the drift parameter, with a focus on linear models. Explicit parameter relationships guide tuning for optimal performance, and a marginal estimator is proposed for $p\gg d$ scenarios under partial orthogonality assumption. Numerical studies confirm the adaptive Lasso’s superiority over standard Lasso and MLE in accuracy and support recovery, providing robust solutions for high-dimensional stochastic processes.  \\

 

\end{abstract}

\tableofcontents

\section{Introduction}
Over the past decades, statistical inference for stochastic processes has garnered increasing attention, driven by their extensive applications across diverse scientific fields. In particular, stochastic differential equations (SDEs) have proven fundamental in disciplines such as biology \cite{Ric77}, epidemiology \cite{Bai57}, physics \cite{Pap95}, economics \cite{Ast21}, neurology \cite{Hol76}, and mathematical finance \cite{Hul00}. This wide applicability has spurred significant advancements in both parametric and non-parametric inference methods under various frameworks.

Simultaneously, the growing importance of high-dimensional data has introduced new complexities to statistical modeling. Researchers have explored scenarios where the number of model parameters far exceeds the available observations or where most parameters exhibit specific asymptotic behavior, departing from the classical approach that assumes only the number of observations grows asymptotically. While substantial progress has been made in understanding high-dimensional frameworks for simpler models \cite{Sara11, Huang08_2, Can19, Gen2024}, the study of high-dimensional stochastic processes remains relatively scarce.

Existing work on high-dimensional diffusions has predominantly focused on particle interaction systems within mean field theory, with notable parametric and non-parametric results explored in \cite{Amo24,Amo23, BPZ24,Bel24, Bis11, Che21, Del23,  Gen21, Gen2024, Liu22, Sha23}, among others. However, most studies have restricted parameter spaces to finite dimensions, leaving the case of infinite-dimensional parameter spaces underexplored. Expanding both theoretical and methodological knowledge at the intersection of high-dimensional frameworks and stochastic processes has thus become a topic of significant scientific interest.

In this article, we examine the statistical analysis of a 
$d$-dimensional ergodic diffusion process, defined as the solution to the stochastic differential equation \begin{equation} \label{model} 
dX_t = -\bt(X_t)dt + dW_t, \hspace{2cm} t \geq 0, 
\end{equation} 
on a filtered probability space $(\Omega, \mathcal{F}, (\mathcal{F}_t)_{t \geq 0}, \P)$. 
Here $\theta \in \Theta \subseteq \R^p$ is  the model parameter and $\Theta$ is an open set. Furthermore, $\bt:\R^d \rightarrow \R^d$ is the drift function of the process and $W = (W_t)_{t \geq 0 }$  denotes a standard $d$-dimensional Brownian motion, and the initial value $X_{0}$ is a random vector independent of $W$.
We consider the scenario where the entire trajectory of the process $(X_t)_{t  \in [0,T]}$ is observed up to time $T$, allowing for asymptotically large values of $p, d $ and $T$.

In high-dimensional statistics, particularly in parametric inference for large-scale systems, a popular approach assumes sparsity in the parameters. Following this perspective, we impose a sparsity constraint on the true parameter $\theta_0$ of the model \eqref{model}:
\begin{equation}\label{sparsityConstraint} 
\|\theta_0\|_0 := \text{card}\left(\left\{ 1 \leq i \leq p |~ \theta_{0}^i \neq 0 \right\}\right) = s. 
\end{equation} 
Under this assumption, regularization methods have emerged as a key tool to improve classical estimators' performance. A penalized estimator of $\theta_0$
is generally defined as 
\begin{equation} 
\widehat{\theta}(\lambda) = \argmin{\theta \in \Theta} \{ L(\theta) + c(\theta , \lambda) \}, 
\end{equation}
where 
$L$ represents a loss function based on the observed data, and $c$
is a penalty function depending on $\theta$ and  a tuning parameter $\lambda$.
The Lasso estimator, where $c(\theta,\lambda)=\lambda \|\theta\|_1$, has been extensively studied for high-dimensional diffusion processes. Key results include 
$l_2$-error bounds and convergence rates under structural assumptions on the diffusion \cite{Pod22,Tro23}, estimation of interaction matrices in Ornstein-Uhlenbeck processes \cite{Cio20}, and extensions to L\'evy-driven processes \cite{Dext22}. Recent studies, such as \cite{pina24}, also address the impact of discrete sampling on the Lasso estimator's performance.

Despite its advantages, the Lasso estimator lacks the oracle property, meaning it cannot consistently select non-zero coefficients or achieve asymptotic normality for non-zero parameters \cite{Fan01, Fan04}. To address these limitations, various estimators have been proposed. For linear regression, the bridge estimator \cite{Huang08} and SCAD penalty \cite{Fan97} have gained popularity for their oracle properties, though their non-convexity presents computational challenges. The adaptive Lasso, introduced in \cite{Zou06}, offers a convex alternative that improves variable selection and reduces estimation bias by applying differential penalties to zero and non-zero coefficients. Studies such as \cite{Huang08_2} have demonstrated the adaptive Lasso's ability to achieve variable selection and asymptotic normality under specific conditions.

Research on the adaptive Lasso in the context of diffusion processes remains relatively sparse. While prior studies, such as \cite{De12} and \cite{Gai19}, investigate variable selection within multidimensional frameworks, a thorough analysis of its variable selection properties in asymptotic settings is still lacking. This study seeks to address this gap by extending existing results and identifying the conditions under which the adaptive Lasso achieves both support recovery and asymptotic normality in estimating the drift parameter of the diffusion process defined in \eqref{model}.

This focus is particularly relevant in light of recent advancements in the field of graphical models for stochastic processes \cite{DDK24, LKS24, MH20, MH22}, where variable selection plays a pivotal role in uncovering the causality structure of the underlying model. To pursue this goal, we concentrate on linear models where the drift function of \eqref{model} assumes the form specified in \eqref{linearModel}. While this choice enhances the clarity of our analysis, the methods and results presented here can be extended to accommodate more complex drift functions, leveraging approaches outlined in works such as \cite{Pod22} and \cite{Tro23}.

Our main contributions, Theorems \ref{constheorem} and \ref{theoremAsNor}, show that under Assumption \hyperref[assumption:B3]{\((\mathcal{B}_3)\)} and condition \eqref{conditionANormality}, the Lasso estimator achieves sign consistency and asymptotic normality. These results explicitly characterize the relationships between model parameters and their asymptotic behavior, providing guidance for selecting tuning parameters and optimizing performance. Key insights are detailed in Remarks \ref{r: ConclusionParameters} and \ref{r: CombinationSC-AN}. For cases where $p\gg d$,
 we propose a marginal estimator (Section \ref{ss: partialOrtho}) that satisfies partial orthogonality assumptions, offering a robust alternative to classical methods in high-dimensional settings. Finally, numerical studies validate our theoretical findings, demonstrating that the adaptive Lasso outperforms standard Lasso and MLE in accuracy and support recovery.

% In general linear models, the drift function $\bt$  of \eqref{model} is given by
% \begin{equation}\label{linearModel}
%     \bt = \phi_0 + \sum_{j=1}^{p} \theta_j \phi_j,
% \end{equation}
% where $\phi_j : \R^d \rightarrow \R^d$  for all $j=0,\dots,p$, and   $\theta_j > 0$ for $j=1,\dots,p.$ 

The paper is organised as follows. In Sections \ref{ss: notation} and \ref{ss: ass}, we introduce the notation, the model, and its primary assumptions. Section \ref{s: main} presents the central results, with a focus on Theorem \ref{constheorem} and Theorem \ref{theoremAsNor}, which establish the adaptive Lasso’s ability to achieve consistent support recovery and asymptotic normality, respectively.
Section \ref{s: preestimators} examines various pre-estimator options tailored to different scenarios involving distinct relationships between the model parameters. In Section \ref{s: conineqsection}, we provide a brief overview of the concentration inequalities applicable depending on the properties of the diffusion process.
Section \ref{s: numericalStudies} validates the theoretical findings through numerical simulations, while Section \ref{s: proofs} contains the main auxiliary results and the detailed proofs of the paper.

\subsection{Notation}{\label{ss: notation}}
Before presenting our results, we first provide a concise overview of the notation used throughout this paper. All vectors are treated as column vectors. For a vector $x \in \R^d$, we denote its $i$-th component by $x^i$ and express its $l_q$-norm for $q \in [1,\infty]$ as:
\begin{align}
    \|x\|_q := \left( \sum_{i=1}^d |x^i|^q\right)^{\frac{1}{q}} \text{for   } q \in [1,\infty),\qquad
    \|x\|_\infty := \max_{1 \leq i \leq d} |x^i|.
\end{align}
The canonical basis vector of $\R^p$ is denoted by $e_i$ for $i = 1, \dots, p$.
Given $x, y \in \R^d$, $\langle x, y \rangle$ represents the scalar product. For any vector $\theta \in \R^p$, its sign vector is denoted by sign($\theta$) $\in \R^p$, defined as sign($\theta$) = ($\sign(\theta^1), \sign(\theta^2), \dots, \sign(\theta^p)$)$^{\star}$, where $\sign(\theta^i) = 1$ if $\theta^i > 0$, $\sign(\theta^i) = -1$ if $\theta^i < 0$, and $\sign(\theta^i) = 0$ if $\theta^i = 0$.

For a set $\mathcal{S} \subset {1, \dots, d}$ and a vector $x \in \R^d$, we denote $x_{\mathcal{S}} \in \R^{\text{card}(\mathcal{S})}$, which contains only the components $x^j$ for $j \in \mathcal{S}$; $\supp^C$ denotes the complement of $\supp$.
Given a matrix $M \in \M{a}{b}$, $M^{\star}$ denotes its transpose, $M_j$ the $j$-th row, and $M^i$ the $i$-th column. For a subset $\mathcal{G} \subset {1, \dots, b}$, $M^{\mathcal{G}}$ is the $a \times |\mathcal{G}|$ submatrix of $M$, consisting of the column vectors $M^j$ for $j \in \mathcal{G}$. The minimum and maximum eigenvalues of $M$ are denoted by $l_{\min}(M)$ and $l_{\max}(M)$, respectively, and the operator norm of $M$ is expressed as $\nop{M}$.

For a function $f: \Theta \times \R^d \rightarrow \R^d$, $\dot{f}$ denotes the derivative of $f$ with respect to the parameter $\theta \in \Theta$. A function $f: \R^d \rightarrow \R^d$ is said to have polynomial growth if $\| f(x) \|_2 \leq A(1 + \|x\|_2^q)$ for some $q, A > 0$. In particular, if $q = 1$, the function is said to have linear growth. For a Lipschitz continuous function $f$, its Lipschitz norm is denoted by $\nlip{f}$.

%We define the Fisher information of the process by 
%\begin{equation}\label{fisherInfo}
        %I(\theta) = \E_{\theta}\left[ \dot{\bt}(X_{t_0})^{\star} \dot{\bt}(X_{t_0})\right].
%\end{equation}
%If we introduce the function $\Phi(x): = (\phi_1(x),\dots,\phi_p(x))\in \R^{d\times p}$ whose $j$-th column is the $\phi_j$ function considered in \eqref{linearModel}.
% It is important to highlight that, under the linear  model, the derivative  of the drift function is independent of $\theta$. Consequently, we can express \eqref{fisherInfo} as \(\E_{\theta}\left[ \Phi(X_{t_0})^{\star} \Phi(X_{t_0}) \right]\). Furthermore, we will denote the Fisher information under the true value of the parameter \(\theta_0\) simply as \(I\).

\subsection{Model under consideration and initial assumptions}{\label{ss: ass}}

%Added on 9/12/2024 

Suppose that the observations $(X_t)_{t\in[0,T]}$ were generated by a linear model, i.e., the stochastic differential equation \eqref{model} with the drift function $\bt$ given by
\begin{equation}\label{linearModel} 
\bt = \phi_0 + \sum_{j=1}^{p} \theta_0^j \phi_j, 
\end{equation}
where $(\phi_j)_{j\geq 0}$ is a given alphabet of  mappings from $\R^d$ to $\R^d$, and $\theta_0 = (\theta_0^1,\dots, \theta_0^p) \in \R^p$ is an unknown parameter to be estimated, satisfying the sparsity constraint \eqref{sparsityConstraint}. For simplicity, we assume $\supp = {1, ..., s}$, i.e., $\theta_0 = (\thetaS^{\star}, \thetaSc^{\star})^{\star}$, where $\thetaS$ is an $s \times 1$ vector of nonzero coordinates, and $\thetaSc$ is a $(p-s) \times 1$ vector of zeros.
Moreover, we denote by $\theta_{0,\supp}^{\text{min}}$ and $\theta_{0,\supp}^{\text{max}}$ the minimum and maximum of the absolute values of $\thetaS$, respectively. We define $\Phi(x): = (\phi_1(x),\dots,\phi_p(x)) \in \R^{d \times p}$, where the $j$-th column corresponds to the function $\phi_j$ considered in \eqref{linearModel}.

Throughout this paper, we consider the following assumptions related to the nature of the process: \\ 

\noindent
\textbf{Assumption ($\mathcal{A}_1$):} The function $b_{\theta}(\cdot)$ is globally Lipschitz, and there exists a constant $M$ such that
\begin{equation}
\left\langle b_{\theta}(x),x \right\rangle \geq M \|x\|_2^{2}.
\end{equation}\label{assumption:A1}

\noindent 
Assumption~\hyperref[assumption:A1]{\((\mathcal{A}_1\))} ensures that equation \eqref{model} has a unique strong solution $X$, which constitutes an ergodic homogeneous continuous Markov process with an invariant distribution (see \cite{Rogers_Williams_2000}, Theorem 12.1). This solution possesses bounded moments of all orders. We further assume that the process and its properties satisfy the necessary conditions for the stability and consistency of the theoretical framework developed in this study.

We make the additional assumption: \\

\noindent
\textbf{Assumption ($\mathcal{A}_2$):} $X_0$ follows the invariant distribution. \label{assumption:A2} \\

\noindent
This assumption ensures that $X$ is the strictly stationary solution of equation \eqref{model}. From the observations $(X_t)_{t \in [0,T]}$, we can deduce that the scaled negative log-likelihood function is expressed as

\begin{equation}\label{LT} 
L_T(\theta) = \frac{1}{T}\int_0^T \bt(X_t)^{\star}dX_t + \frac{1}{2T}\int_0^T \bt(X_t)^{\star}\bt(X_t) dt. \end{equation}
Additionally, we define the empirical covariance matrix:
\begin{equation}\label{CT} 
C_T: = \frac{1}{T}\int_0^T \Phi(X_t)^{\star}\Phi(X_t) dt, 
\end{equation}
which serves as the analogue of the Gram matrix in linear regression models. This matrix satisfies:
\begin{equation}\label{Cinfty} 
C_{\infty}:=\E[C_T]  = \E[\Phi(X_0)^{\star}\Phi(X_0)]. 
\end{equation}
Denoting $\supp = $supp($\theta_0$), we introduce the submatrices:
\begin{equation} 
C_T^{\supp \supp}: = \frac{1}{T}\int_0^T {\Phi^{\supp}(X_t)}^{\star}\Phi^{\supp}(X_t) dt \hspace{0.3cm} \text{and} \hspace{0.3cm} C_T^{\supp^C \supp}: = \frac{1}{T}\int_0^T {\Phi^{\supp^C}(X_t)}^{\star}\Phi^{\supp}(X_t) dt, \end{equation}
along with their means $C_\infty^{\supp \supp}: = \E[C_T^{\supp \supp}]$ and $C_\infty^{\supp^C \supp}: = \E[C_T^{\supp^C \supp}]$.

The quantity $C_T^{\supp \supp}$, referred to as the active block, represents the covariance of the relevant functions $\phi_j$ in \eqref{linearModel} that correspond to the non-zero coefficients $\theta_0^j$. Conversely, $C_T^{\supp^C \supp}$ captures the correlation between the relevant and irrelevant functions $\phi_j$ in the model.
To quantify these terms, we define the following quantities:
\begin{align}
    \mathfrak{L} := \max_{ 1 \leq j \leq p-s} \left\|(C_{\infty}^{\supp^C\supp})_j\right\|_2, \qquad
    \mathfrak{M}:=\max_{1\leq j \leq p} \left|C_{\infty}^{jj}\right|,
\end{align}
with their roles discussed in Remark \ref{r: quantitiesModel}. Another critical component in the analysis is the martingale term:
\begin{equation} 
\epsilon_T := \frac{1}{T}\int_0^T \Phi(X_t)^{\star}dW_t , 
\end{equation}
which is intrinsically linked to the stochastic nature of the process.

We further impose the following assumption: \\

\noindent 
\textbf{Assumption ($\mathcal{A}_3$):} We assume that
\begin{equation}
    I^{\supp}: = \E\left[\Phi^{\supp}(X_0)^{\star}\Phi^{\supp}(X_0)\right],
\end{equation}\label{assumption:A3}
 is positive definite, and we denote by $\tau_{\min} >0$ its minimum eigenvalue.

\begin{remark}\label{remarkHD} \rm 
Assumption~\hyperref[assumption:A3]{\((\mathcal{A}_3\))} can be regarded as an analogue of the identifiability condition for parameters in classical settings. Similar assumptions have been examined in \cite{Pod22, Gai19}, and \cite{Str21}, where a more stringent requirement was imposed by assuming that the covariance matrix of the process is positive definite $\P$-a.s. However, our approach significantly relaxes this condition. Instead of demanding positive definiteness for the entire Fisher information matrix $I$, we restrict this requirement to the structure directly associated with the support of $\theta_0$. This nuanced adjustment allows us to address processes with degenerate Fisher information structures, thereby encompassing a broader range of diffusion processes with complex characteristics.

This relaxation, which imposes the positive definiteness condition only on $I^{\supp}$ while permitting $I$ to be degenerate, is particularly relevant in high-dimensional statistical contexts. By loosening the assumption on $I$, we eliminate the need to assume $p \leq d$. This flexibility enables our results to apply to a more general framework of high-dimensional inference, where the parameter dimension can exceed the process dimension. \qed
\end{remark}

\begin{remark}\rm \label{r: quantitiesModel}
The quantity $\mathfrak{L}$ serves as an analogue of the normalization constant in regression settings. Specifically, it ensures that $\theta_0$ is selected so that the rows of $C_T$ are standardized with respect to the $\ell_2$ norm.

In contrast, $\mathfrak{M}$ represents the maximum of the expected autocorrelation among the model variables. The main purpose of defining these quantities is to explicitly track their behavior, as $\mathfrak{L}$ and $\mathfrak{M}$ may depend on critical model parameters such as $p$, $d$, $s$, or $T$. By doing so, we effectively capture how these parameters influence the results. Consequently, quantities like $\mathfrak{L}$ and $\mathfrak{M}$ (along with others, such as $\tau_{\min}$ or $\mathcal{K}$) are explicitly integrated and monitored throughout the analysis. \qed
\end{remark}

\section{Main results}{\label{s: main}}
% During our work, we consider the stochastic process described by \eqref{model} with drift function of the type \eqref{linearModel}, where $\theta_0$ in the underlying true parameter in $\R^p$ satisfying the sparse constraint \eqref{sparsityConstraint}. 

% By Girsanov's theorem, when the whole trajectory of the process is observed, we can deduce the negative log-likelihood function of the drift parameter, which is given by
% \begin{equation}\label{MLE}
%     L_T(\theta) = \frac{1}{T}\int_0^T \bt(X_t)^{\star}dX_t + \frac{1}{2T}\int_0^T \bt(X_t)^{\star}\bt(X_t) dt.
% \end{equation}
The primary goal of this paper is to present statistical results concerning the performance of the adaptive Lasso estimator for the drift function, particularly in the task of support recovery. The adaptive Lasso estimator of $\theta_0$ is defined as
\begin{equation}\label{definitiAdaptiveLasso}
    \alasso := \argmin{\theta \in \Theta} \{ L_T(\theta) + \lambda\| w^{\star}\theta \|_1\},
\end{equation} 
where $w_j = |\widetilde{\theta}^j|^{-1}$ is a $p$-dimensional weight vector derived from the pre-estimator $\widetilde{\theta}$. A common approach to demonstrate that an estimator can recover the support of the true parameter, especially in high-dimensional settings, is to show that the estimator is sign consistent with the parameter.

\begin{definition}\label{definitionSignConsistency} (Sign consistency) 
We say that two vectors, $\theta$ and $\widehat{\theta}$, both in $\R^p$, are sign consistent, denoted as $\theta =_s \widehat{\theta}$, if sign$(\theta) = $sign$(\widehat{\theta})$. 
\end{definition}

\noindent
Therefore, the central objective of this article is to establish conditions, together with
Assumptions~\hyperref[assumption:A1]{\((\mathcal{A}_1)\)}-\hyperref[assumption:A3]{\((\mathcal{A}_3)\)},  that guarantee the sign consistency of the adaptive Lasso estimator $\alasso$ for the parameter $\theta_0$.

We consider the following assumptions: \\

%{\color{red} Mark: Here it would be fair to mention that use similar conditions as for liner regression, adapted to our setting}\\

\noindent \textbf{Assumption ($\mathcal{B}_1$):}\label{assumption:B1} The initial estimator $\widetilde{\theta}$ is $r_T$-consistent for the estimation of certain unknown vector of constants $\eta_0$ depending on $\theta_0$, i.e. 

\begin{equation}\label{rTcons}
    r_T \|\widetilde{\theta} - \eta_0\|_{\infty} = O_{\mathbb{P}}(1),\quad T\to\infty, 
\end{equation}
and for some constants $\mathcal{M}_{1,T}$ and $\mathcal{M}_{2,T}$ the vector $\eta_0$ satisfies 

\begin{equation}
    \max_{j \notin \supp} |\eta_0^j| \leq \mathcal{M}_{2,T} \quad \text{and} \quad \left( \sum_{j\in \supp}\left( \frac{1}{|\eta_0^j|} + \frac{\mathcal{M}_{2,T}}{|\eta_0^j|^2} \right)^2 \right)^{\frac{1}{2}} \leq \mathcal{M}_{1,T} = o(r_T), \quad T\to\infty. 
\end{equation}

\smallbreak
\noindent \textbf{Assumption ($\mathcal{B}_2$):}\label{assumption:B2} For $u=(|\eta_0^j|^{-1}\sign(\theta_0^j), j \in \supp)^{\star}$ the adaptive irrepresentable condition holds for some $\kappa<1$, i.e. 

\begin{equation}\label{irrepresent}
\left| \left( C_\infty^{\supp^C \supp} (C_\infty^{\supp\supp })^{-1} u \right) ^{j-s} \right|\leq \frac{\kappa}{|\eta_0^j|}, \quad j\notin\supp. 
\end{equation}

\smallbreak
\noindent \textbf{Assumption ($\mathcal{B}_3)$:}\label{assumption:B3} The parameters of the model satisfy $\mathfrak{L}/\tau_\text{min}>\tau_0>0$, where $\tau_0$ is fixed and 

% $\left\{p,d,T,\lambda,s,b_{\min}\right\}$ satisfy

% SEt B1
\begin{equation}
 \left( \left(\frac{1}{(\theta_{0,\supp}^\text{min})^2}+ \frac{ (\mathcal{M}_{2,T} + \frac{1}{r_T})^2(\frac{\tau_{\text{min}}}{\sqrt{s}}+2\mathfrak{L})^2}{\lambda^2}\right)(\sqrt{\mathcal{K}}+\mathfrak{M})+ \mathcal{K}s  \right) \frac{s\log{p}}{\tau_{\text{min}}^2 T} \to 0
\end{equation}
and
%Set B2
\begin{equation}
\frac{\log{p}}{T}+\left( \frac{1}{s} + \frac{\mathfrak{L^2}}{\tau_\text{min}^2} \right) \frac{\mathcal{K}s^2\log{p}}{\tau_{\text{min}}^2 T}\mathcal{M}_{1,T}^2+\frac{\lambda\mathcal{M}_{1,T}}{\theta_{0,\supp}^\text{min}\tau_{\text{min}}}\to 0.
\end{equation}

%Sets B3 and B4
% \begin{equation}
%  \left(  \frac{ (\sqrt{\mathcal{K}}+\mathfrak{M})(\mathcal{M}_{2,T} + \frac{1}{r_T})^2(\tau_{\text{min}}+2\mathfrak{L})^2}{\lambda^2} +\mathcal{K}\right)\frac{s\log{p}}{\tau_{\text{min}}^2 T}\to 0.
% \end{equation}

%     \begin{equation}
%          \sqrt{\frac{\ln(s) \vee \ln(p-s)}{T}} + \frac{1}{b_{\min}}\sqrt{\frac{ds\ln(s)}{T}} + \frac{\lambda \sqrt{s}}{b^2_{\min}} +  \frac{sd}{\lambda r_n} \sqrt{\frac{d\ln(p-s)}{T}} + \frac{sd}{b_{\min} r_n}  \rightarrow 0 
%     \end{equation}
% \hl{Need to change after the update of the proof}

% \begin{remark} \rm 
% \hl{Need to add comment about the interpretation of all of the assumptions.} 
% \end{remark}

\smallbreak
\noindent \textbf{Assumption ($\mathcal{C}$):}\label{assumption:C}  For the random functional
\begin{equation}\label{Hvdef}
    H_v(X_{[0,T]}) := v^{\star} \left( \frac{1}{T}\int_0^T \left(\Phi(X_t)^{\star}\Phi(X_t) - \E_{\theta_0}\left[ \Phi(X_t)^{\star}\Phi(X_t) \right] dt \right) \right)v, 
\end{equation}
there exists a constant $\mathcal{K}$ such that that for all $\mu \in \R$  and for all $v\in \R^p$ satisfying $\|v\|_2 = 1$
    \begin{equation}\label{conineq}
        \E\left[\exp\left(\mu H_v(X_{[0,T]}) \right)\right] \leq \exp\left(\mathcal{K} \mu^2 /T\right).
    \end{equation}

\smallbreak

\begin{remark}\label{conditions} \rm 
The assumptions outlined above build upon those introduced in \cite{Zou06} for the linear regression setting, but with key modifications tailored to the framework of diffusion processes.
In linear regression, the Gram matrix is pivotal in characterizing the dependence structure of the covariates. In contrast, within our diffusion process framework, this role is assumed by the theoretical Fisher information matrix $C_\infty$, 
which encapsulates parameter information while accounting for the stochastic and temporal dependencies intrinsic to diffusion processes. This substitution is fundamental, reflecting the transition from static regression models to the dynamic nature of diffusion models.

Additionally, our methodology tracks a more extensive set of parameters, allowing us to establish sharper bounds on the maximal number of non-zero elements in the unknown parameter vector $\theta_0$.
This enhancement is particularly significant in high-dimensional settings, where managing the interplay between sparsity $s$, dimensionality 
$p$, and other model characteristics is crucial to achieving both robust theoretical results and practical feasibility.
These extensions not only generalize the assumptions in \cite{Zou06} but also adapt them to address the unique challenges of high-dimensional diffusion models, ensuring their relevance and applicability in this more complex context.
\qed
\end{remark}

\begin{remark}\label{remarkAdIC} \rm 
Assumption \hyperref[assumption:B1]{\((\mathcal{B}_1)\)}  ensures that the initial estimator $\widetilde{\theta}$ provides a reliable approximation of a proxy $\eta_0$ for the true parameter $\theta_0$. The adaptive irrepresentable condition \hyperref[assumption:B2]{\((\mathcal{B}_2)\)} further defines the relationship between \(\eta_0\) and \(\theta_0\) by establishing a balance between the elements of \(\eta_0\) within the support \(\supp\) of \(\theta_0\) and those outside it. This balance is critical for guaranteeing the success of the Adaptive Lasso in identifying the correct support and accurately estimating the non-zero components of \(\theta_0\), since this condition ensures that the weights $w_j$ are properly scaled: they are neither excessively large for indices
$j \in \supp$ (the support of $\theta_0$) nor excessively small for $j \notin \supp$. 

In cases where $\sign{(\eta_0)} = \sign{(\theta_0)}$, meaning the initial estimator is sign consistent with a convergence rate $r_T$, the adaptive irrepresentable condition outlined in Assumption \hyperref[assumption:B2]{\((\mathcal{B}_2)\)} is automatically satisfied, and the term $\mathcal{M}_{2,T}$ in Assumption \hyperref[assumption:B1]{\((\mathcal{B}_1)\)} vanishes.
For further details and examples of potential initial estimators that fulfill these assumptions, please refer to Section \ref{s: preestimators}.
\qed
\end{remark}

\begin{remark} \rm 
Assumption  \hyperref[assumption:C]{\((\mathcal{C})\)}  is a standard condition in this context, serving to control the behavior of the noise. This restriction is crucial for deriving the required concentration inequalities. For further details, including examples that guarantee the validity of this assumption and guidance on selecting the constant $\mathcal{K}$ based on the specific model, we refer to Section 
\ref{s: conineqsection} and the references therein.\qed
\end{remark}

\begin{remark}\label{r: ConclusionParameters} \rm 

% Assumption ($\mathcal{B}_3$) restricts the number of zero and nonzero coordinates in the parameter $\theta_0$ depending on the choice of the tuning parameter $\lambda$ and given specifications of the particular model. The number of coordinates permitted is also defined by the tail behavior of the noise terms defined by the Assumption ($\mathcal{C}$). 
% We often have that the convergence rate $r_T = O(\sqrt{T/(s\log(p))}$ and $\lambda = 1 / T^a$ for some $0<a<1$. In this case the total number of coordinates $p$ can be as large as $\exp{(T^{\delta})}$ for some $0<\delta<1$, but the number of nonzero coordinates allowed is of the order $\min\left( T^{2a}, T^{\frac{1-\delta}{2}} \right)$, assuming that $1/\theta_{0,\supp}^\text{min} + \mathcal{K}+\mathfrak{M}= O(1)$, $\mathcal{M}_{1,T} = O(\sqrt{s})$ and $\mathfrak{L}/\tau_\text{min} = O(1/\sqrt{s})$. Thus even though the total number of coordinates of $\theta_0$ can be very large, the number of nonzero coordinates have to be comparable to the time of observations $T$. 

Assumption  \hyperref[assumption:B3]{\((\mathcal{B}_3)\)}  places constraints on the number of zero and nonzero coordinates in the parameter $\theta_0$, which depend on the choice of the tuning parameter $\lambda$ and the specific characteristics of the model. The number of permissible coordinates is further influenced by the tail behavior of the noise terms, as outlined in Assumption \hyperref[assumption:C]{\((\mathcal{C})\)}. Typically, the convergence rate of the pre-estimator is $r_T = O(\sqrt{d {T}/({s \log(p)})})$ (cf. Section \ref{sec3.1}), and the tuning parameter is selected as $\lambda = (dT)^{-a}$ for some $0 < a < 1$. With these settings, the total number of coordinates $p$ can grow as large as $\exp((dT)^{\delta})$ for some $0 < \delta < 1$. However, the number of nonzero coordinates is constrained to be of the order
\[
\min\left\{(dT)^{2a}, (dT)^{\frac{1-\delta}{2}}\right\},
\]
under the assumptions that $1 / \theta_{0,\supp}^\text{min} + \mathcal{K} + \mathfrak{M} = O(1)$, $
\mathcal{M}_{1,T} = O(\sqrt{s})$, and $\mathfrak{L} / \tau_\text{min} = O(1 / \sqrt{ds})$.

This implies that, while the total number of coordinates in $\theta_0$ can be very large, the number of nonzero coordinates must remain roughly proportional to the duration of the observation period $T$ and the dimension $d$ of the observed process.
\qed
\end{remark}

\noindent 
The next theorem demonstrates the sign consistency of the estimator $\widehat{\theta}$. 

\begin{theorem}\label{constheorem}
    Assume~\hyperref[assumption:A1]{\((\mathcal{A}_1)\)}-\hyperref[assumption:A3]{\((\mathcal{A}_3)\)},~\hyperref[assumption:B1]{\((\mathcal{B}_1)\)}-\hyperref[assumption:B3]{\((\mathcal{B}_3)\)} and \hyperref[assumption:C]{\((\mathcal{C})\)}  hold. Then we have that 
    \begin{equation}
        \P(\alasso =_s \theta_0) \rightarrow 1, \qquad \text{as } T\to\infty. 
    \end{equation}
\end{theorem}

\begin{proof} It follows from the Karush-Kuhn-Tucker conditions that $\alasso$ is a solution of the adaptive Lasso estimator \eqref{definitiAdaptiveLasso} if $\nabla_{\theta}L_T(\alasso) + \lambda w^{\star}\beta(\alasso) = 0$, where $\beta(\alasso)\in \R^p$ is given as

\begin{equation}\label{bvector}
\begin{cases}
  \beta(\alasso)^i = \sign(\alasso^i) & \text{if}\quad \alasso^i \neq 0\\[1.5 ex]
\beta(\alasso)^i \in (-1,1) & \text{if}\quad \alasso^i = 0.
\end{cases}  
\end{equation}

% $\beta(\theta)$ as 
% \begin{equation}\label{bvector}
% \beta^i(\theta) = 
% \begin{cases} 
% \sign(\theta^i) & \text{if } i \in \text{supp}(\theta)  \\ 
% [-1,1] & \text{if } \theta_i = 0. \\ 
% \end{cases}
% \end{equation}
\noindent Since $\nabla_{\theta}L_T(\alasso) = C_T(\alasso - \theta_0) + \epsilon_{T}$, 
we have that for all $j=1,\dots,p$: 
% \begin{equation}
%     (C_T)_j(\alasso - \theta_0) + \epsilon_{T}^j = -\lambda w^j \beta^j(\theta).
% \end{equation}
\begin{equation}
    (C_T(\alasso - \theta_0))^j + \epsilon_{T}^j = -\lambda w^j \beta(\alasso)^j.
\end{equation}
% where $(C_T)_j$ means the $j$-th row  of $C_T$ transposed, i.e. $(C_T)_j = \left(\frac{1}{T}\int_0^T \phi_j(X_t)^{*}\Phi(X_t)dt\right)^{\star}.$  
From \eqref{bvector}, we deduce: 
% \begin{equation}
%     \begin{cases}
%         (C_T)_j(\alasso - \theta_0) + \epsilon_{T}^j = -\lambda w^j \sign(\alasso^j) \hspace{0.4cm} &\text{if   }  j \in \text{supp}(\alasso)\\
%         |(C_T)_j(\alasso - \theta_0) + \epsilon_{T}^j| \leq  \lambda w^j \hspace{0.4cm} &\text{if   }  j \notin \text{supp}(\alasso).
%     \end{cases}
% \end{equation}
\begin{equation}
    \begin{cases}
        (C_T(\alasso - \theta_0))^j + \epsilon_{T}^j = -\lambda w^j \sign(\alasso^j) \hspace{0.4cm} &\text{if   }  j \in \text{supp}(\alasso)\\[1.5 ex]
        |(C_T(\alasso - \theta_0))^j + \epsilon_{T}^j| \leq  \lambda w^j \hspace{0.4cm} &\text{if   }  j \notin \text{supp}(\alasso).
    \end{cases}
\end{equation}
% Let us define $\widehat{u} \in \R^s$ such that $\widehat{u}^i = |\lasso^i|^{-1}$sign($\theta_0^i$), and recalling that $\supp$ is the support of $\theta_0$; to guarantee that the sing-consistency property 
Let us define $\widetilde{u}:=(w^j \sign(\theta_0^j), j\in \supp)^{\star} \in \R^s$.  
% such that $\widehat{u}^i = |\lasso^i|^{-1}$sign($\theta_0^i$), 
% and recalling that $\supp$ is the support of $\theta_0$; 
Then, the necessary condition to guarantee  the sing consistency property (\ref{definitionSignConsistency})  for $\alasso$ becomes  
% \begin{equation}\label{eq1}
%     C_T^{\supp\supp }(\alasso - \theta_0)|_{\supp} + \epsilon_T|_{\supp} = -\lambda \widehat{u}
% \end{equation}
\begin{equation}\label{eq1}
    C_T^{\supp\supp }(\alasso_\supp - \thetaS) + \epsilon_{T,\supp} = -\lambda \widetilde{u}
\end{equation}
and 
% \begin{equation}
%     |\left(C_T^{\supp^{C}\supp}(\alasso - \theta_0)|_{\supp} + \epsilon_T|_{\supp^C}\right)^j| \leq  \lambda w|_{\supp^C}^j, \quad \text{for} \quad j=1, ..., p-s. 
% \end{equation}
\begin{equation}
    \left|\left(C_T^{\supp^{C}\supp}(\alasso_\supp - \thetaS)\right)^j + \epsilon_{T,\supp^C}^j\right| \leq  \lambda w_{\supp^C}^j, \quad \text{for} \quad j=1, ..., p-s. 
\end{equation}
% since $(\alasso - \theta_0)|_{\supp^C}$ = 0.
Therefore, for $\alasso=_s \theta_0$ to be satisfied, it is necessary that   
\begin{equation}\label{conditionsforeachi}
    \begin{cases}
        \sign(\theta_{0}^j)( \alasso^j - \theta_0^j) \leq |\theta_0^j|, \hspace{0.4cm} &\text{for }j \in \supp \\[1.5 ex]
\left|\left(C_T^{\supp^C \supp}(\alasso_\supp - \thetaS)\right)^{j-s} + \epsilon_T^j\right| \leq  \lambda w^j, \hspace{0.4cm} &\text{for }j \notin \supp.
    \end{cases}
\end{equation}
From \eqref{eq1}, we obtain 
\begin{equation} \label{eq3}
    \alasso_\supp = \thetaS- (C_T^{\supp\supp })^{-1}\left(\lambda \widetilde{u} + \epsilon_{T, \supp}\right).
\end{equation}
 Combining \eqref{conditionsforeachi} and \eqref{eq3}, we observe that 
$\alasso \neq_s \theta_0$ if there 
 exists $j \in \supp$ such that
    \begin{equation}
         \left|\left( (C_T^{\supp\supp })^{-1}\left(\lambda \widetilde{u} + \epsilon_{T,\supp}\right) \right)^j \right| > |\theta_0^j|,
    \end{equation}
   or  $j \notin \supp$ satisfying
   \begin{equation}
       \left|\left(C_T^{\supp^C \supp}(C_T^{\supp\supp })^{-1}\left(\lambda \widetilde{u} + \epsilon_{T,\supp}\right)\right)^{j-s} + \epsilon_T^j\right| >  \lambda w^j.
       % \quad \left|(C_T^{\supp^C \supp})_j (C_T^{\supp\supp })^{-1}\left[\lambda \widetilde{u} + \epsilon_T|_{\supp}\right] + \epsilon_j\right| \geq  \lambda w_j .
   \end{equation}
Thus, for $  0 \leq \kappa \leq 1$, if we define the  sets
\begin{align}
    &B_1 = \left\{\exists j \in \supp :  \left|
    \left((C_T^{\supp\supp })^{-1} \epsilon_{T,\supp}\right)^j\right| \geq \frac{|\theta_0^j|}{2}\right\},  \\
    &B_2 = \left\{\exists j \in \supp :  \left|
    \left((C_T^{\supp\supp })^{-1} \widetilde{u}\right)^j\right| \geq \frac{|\theta_0^j|}{2 \lambda}\right\}, \\
    % \quad \left\{\exists j \in \supp \Big|  \left|(C_T^{\supp\supp })_j^{-1} \widetilde{u}\right| \geq \frac{|\theta_0^j|}{5\lambda}\right\}, \\
    &B_3 = \left\{\exists j \notin \supp :  \left| \left( C_T^{\supp^C \supp}(C_T^{\supp\supp })^{-1} \epsilon_{T, \supp} \right)^{j-s}\right| \geq \frac{1-\kappa - \varepsilon}{2}\lambda w^j\right\},\\
    &B_4 = \left\{\exists j \notin \supp :  \left| \epsilon_T^j\right| \geq \frac{1-\kappa - \varepsilon}{2}\lambda w^j\right\},\\
    &B_5 = \left\{\exists j \notin \supp :  \left| \left( C_T^{\supp^C \supp} (C_T^{\supp\supp })^{-1}\widetilde{u}\right)^{j-s} \right| \geq (\kappa + \varepsilon) w^j\right\},
\end{align}
\noindent it is clear that  $\P(\alasso \neq_s\theta_0) \leq \sum_{i=1}^5 \P(B_i),$ and the statement of the theorem follows immediately from Lemmas \ref{LemmaB1} and \ref{LemmaB2}-\ref{LemmaB5}. 
\end{proof}

\noindent
The next result demonstrates the asymptotic normality and efficiency of the estimator $\alasso_\supp$.

\begin{theorem}\label{theoremAsNor}
   Assume~\hyperref[assumption:A1]{\((\mathcal{A}_1)\)}-\hyperref[assumption:A3]{\((\mathcal{A}_3)\)},~\hyperref[assumption:B1]{\((\mathcal{B}_1)\)}-\hyperref[assumption:B3]{\((\mathcal{B}_3)\)} and \hyperref[assumption:C]{\((\mathcal{C})\)}  hold, and let $\mathfrak{s}_T^2: = \alpha^{\star}(C_\infty^{\supp\supp})^{-1}\alpha$ for any vector $\alpha\in\R^s$ satisfying $\| \alpha \|_2\leq 1$. Then 
    
   \begin{equation}\label{astheoremstat}
    \sqrt{T} \mathfrak{s}_T^{-1} \alpha^{\star}(\alasso_\supp-\theta_{0,\supp})=\sqrt{T} \mathfrak{s}_T^{-1} \alpha^{\star} (C_\infty^{\supp\supp})^{-1}\epsilon_{T,\supp} + o_\P(1) \to_{D} \mathcal{N}(0,1), 
\end{equation}
where $o_\P(1)$ is a term that converges to zero in probability uniformly with respect to $\alpha$, if 

\begin{equation}\label{conditionANormality}
    \frac{\lambda^2 T  \mathcal{M}_{1,T}^2}{\mathfrak{s}_T^{2}\tau_\text{min}^2} + \frac{s\log{s}}{\sqrt{T} \mathfrak{s}_T}\left( \sqrt{\mathcal{K}}+\mathfrak{M} +\frac{s \mathcal{K}}{\tau_\text{min}^4} \right)\to 0. 
\end{equation}

\end{theorem}

\begin{proof}

    %First: By precedent Theorem, we have the A. Lasso is sign-consistent (under  assumptions). Moreover, for the beta-min condition, the values of $\widehat{\theta}_{\supp}$ are away from 0. Therefore, the estimator satisfies the stationary equation $\nabla_{\theta_{\supp}}(L_{T}(\widehat{\theta}_{\supp},\widehat{\theta}_{\supp^C}) + \lambda\|w^{\star}(\widehat{\theta}_{\supp},\widehat{\theta}_{\supp^C} \|_1)=0$. Computing we have
    %\begin{align}
     %   &\frac{1}{T}\int_0^T \Phi^{\supp}(X_t)^{\star}\Phi^{\supp}(X_t)dt(\widehat{\theta}_{\supp} - \theta_{0,\supp}) + \\
      %  &\frac{1}{T}\int_0^T \Phi^{\supp}(X_t)^{\star}\Phi^{\supp^C}(X_t)dt(\widehat{\theta}_{\supp^C} - \theta_{0,\supp^C}) + \\
       % &\frac{1}{T}\int_0^T \Phi^{\supp}(X_t)^{\star}dW_t +\\
        %& \lambda \diag(w_\supp)\sign(\widehat{\theta}_{\supp}) = 0
    %\end{align}

It follows from the Karush-Kuhn-Tucker conditions that $\alasso$ satisfies the identity
\begin{equation}
    (C_T(\alasso - \theta_0))^j + \epsilon_{T}^j = -\lambda w^j \beta^j(\alasso).
\end{equation}
for all $j=1,...,p$, where the vector $\beta(\alasso)$ has been introduced in the previous proof. Recalling that $\supp = \text{supp}(\theta_0)$ and considering the vector $y \in \R^p$ such that $y^i = w^i \beta(\alasso)^i$, if we represent $\alasso = (\alasso_\supp^{\star},\alasso_{\supp^C}^{\star})^{\star}$, the condition reads as

\begin{equation}
    C_T^{\supp\supp}(\alasso_\supp - \theta_{0,\supp}) + C_T^{\supp\supp^C}\alasso_{\supp^C} +\epsilon_{T,\supp} = -\lambda y_\supp,
\end{equation}
and for any $\alpha\in \R^s$ such that $\| \alpha\|_2\leq 1$, we have 
\begin{gather}
     \sqrt{T} \mathfrak{s}_T^{-1}\alpha^{\star}(\alasso_\supp - \theta_{0,\supp})= - \sqrt{T} \mathfrak{s}_T^{-1}\alpha^{\star} (C_T^{\supp\supp})^{-1} C_T^{\supp\supp^C}\alasso_{\supp^C} \\
     - \sqrt{T} \mathfrak{s}_T^{-1}\alpha^{\star} (C_T^{\supp\supp})^{-1}\epsilon_{T,\supp}  -  \sqrt{T} \mathfrak{s}_T^{-1}\lambda \alpha^{\star} (C_T^{\supp\supp})^{-1}y_\supp. 
\end{gather}
According to the Theorem \ref{constheorem}, under given assumptions, the probability $\P(\alasso_{\supp^C}\neq 0)$ converges to zero and, consequently, $ \sqrt{T} \mathfrak{s}_T^{-1}\alpha^{\star} (C_T^{\supp\supp})^{-1} C_T^{\supp\supp^C}\alasso_{\supp^C}$ converges to zero in probability. 

%\begin{equation}
 %   \P\left( \left|\alpha^{\star} (C_T^{\supp\supp})^{-1} C_T^{\supp\supp^C}\alasso_{\supp^C}\right| > \varepsilon \right) \leq \P(\alasso_{\supp^C}\neq 0) \to 0, 
%\end{equation}

Next, for any $\varepsilon>0$, we have 
\begin{gather}
    \P\left( \left|  \sqrt{T} \mathfrak{s}_T^{-1}\lambda \alpha^{\star} (C_T^{\supp\supp})^{-1}y_\supp \right| >\varepsilon \right) \leq \P\left(  \sqrt{T} \mathfrak{s}_T^{-1}\lambda \nop{(C_T^{\supp\supp})^{-1}} \| w \|_2 > \varepsilon \right)  \\
    \leq 
    \P\left(  \frac{\| w \|_2}{\mathcal{M}_{1,T}} > \frac{\varepsilon \tau_\text{min}}{2 \lambda  \sqrt{T} \mathfrak{s}_T^{-1}\mathcal{M}_{1,T}} \right)+\P\left( \nop{(C_T^{\supp\supp})^{-1}}>\frac{2}{\tau_\text{min}} \right),
\end{gather}
and both probabilities converge to zero under \eqref{conditionANormality}  due to Proposition \ref{propBoundInverse} and Lemma \ref{uhat}.  Thus, $\sqrt{T} \mathfrak{s}_T^{-1}\lambda \alpha^{\star} (C_T^{\supp\supp})^{-1}y_\supp$ also converges to zero in probability. Finally, consider 
\begin{gather}
    \P \left(\left| \sqrt{T} \mathfrak{s}_T^{-1} \alpha^{\star} \left((C_T^{\supp\supp})^{-1}-(C_\infty^{\supp\supp})^{-1}\right) \epsilon_{T,\supp}\right|>\varepsilon \right) \nonumber \\
    \leq  \P \left(\sqrt{T} \mathfrak{s}_T^{-1}\left\|(C_T^{\supp\supp})^{-1}-(C_\infty^{\supp\supp})^{-1}\right\|_\text{op}  \|\epsilon_{T,\supp}\|_2>\varepsilon \right) \nonumber \\ 
    \leq  \P \left(\|\epsilon_{T,\supp}\|_2  >\varepsilon\frac{\sqrt{ \mathfrak{s}_T}}{T^{\frac{1}{4}}} \right)
    +\P \left(\left\|(C_T^{\supp\supp})^{-1}-(C_\infty^{\supp\supp})^{-1}\right\|_\text{op}  >\frac{\sqrt{ \mathfrak{s}_T}}{T^{\frac{1}{4}}} \label{equationAN} \right).
\end{gather}
Using Proposition \ref{propMTGbernstein}, we obtain that 
\begin{equation}\label{boundinAN}
    \P \left(\|\epsilon_{T,\supp}\|_2  >\varepsilon\frac{\sqrt{ \mathfrak{s}_T}}{T^{\frac{1}{4}}} \right) \leq 
    2 s\exp\left(\frac{-\sqrt{T}\varepsilon^2 \mathfrak{s}_T}{2s(\sqrt{\mathcal{K}}+\mathfrak{M})}\right) +  6s\exp\left(-\frac{T}{36}\right).
\end{equation}
Applying the resolvent identity and using Proposition \ref{propBoundInverse}, we can bound the second term of \eqref{equationAN} as
\begin{gather}
\P \left(\left\|(C_T^{\supp\supp})^{-1}-(C_\infty^{\supp\supp})^{-1}\right\|_\text{op}  >\frac{\sqrt{ \mathfrak{s}_T}}{T^{\frac{1}{4}}}\right) \nonumber \\
    \leq \P \left(\left\|C_T^{\supp\supp}-C_\infty^{\supp\supp}\right\|_\text{op}  >\frac{\tau_\text{min}^2}{2}\frac{\sqrt{ \mathfrak{s}_T}}{T^{\frac{1}{4}}} \right)
    +\P \left(\left\|(C_T^{\supp\supp})^{-1}\right\|_\text{op}  >\frac{2}{\tau_\text{min}} \right) \nonumber \\
     \leq 6s^2\exp\left(\frac{-\sqrt{T}  \mathfrak{s}_T \tau_\text{min}^4}{36\mathcal{K} s^2}\right)
    +12s^2\exp\left(\frac{-T \tau_\text{min}^2}{144\mathcal{K} s^2}\right) \label{eqqAN}.
\end{gather}
Again, from \eqref{conditionANormality}, the terms in \eqref{boundinAN} and \eqref{eqqAN} converge to 0, which proves the equality in \eqref{astheoremstat}. The statement of the theorem then immediately follows from the limit theorem for multivariate martingales (Theorem 2.1 in \cite{kuchlersorensen}). 
\end{proof}

\begin{remark}\label{r: CombinationSC-AN}\rm
The Lasso estimator is well-known for its inability to achieve both sign consistency and asymptotic normality simultaneously when using a single tuning parameter $\lambda$. This limitation arises because the tuning parameter that ensures accurate support recovery typically fails to provide unbiased or efficient estimates, both of which are essential for asymptotic normality. However, as shown in the theorem, the adaptive Lasso estimator addresses this issue. By utilizing a tuning parameter $\lambda$ that satisfies both Assumption \hyperref[assumption:B3]{\((\mathcal{B}_3)\)} and \eqref{conditionANormality}, the adaptive Lasso successfully balances these objectives, enabling sign consistency and asymptotic normality to coexist. This underscores a significant advantage of the adaptive Lasso in high-dimensional statistical inference.
\qed
\end{remark}

\begin{remark}\label{r: nonlinear-case}\rm 
The linear structure of the drift function, as defined in \eqref{linearModel}, is central to the results presented in this paper. This structure simplifies the Karush-Kuhn-Tucker optimization conditions, reducing them to a straightforward linear form. This simplification is pivotal, as it facilitates the theoretical analysis and enables precise derivations of support recovery guarantees and error bounds.

While the linear drift assumption is a cornerstone of this work, the methodology could be extended to accommodate more complex nonlinear drift structures, akin to the framework explored in \cite{Pod22}. In the nonlinear setting, the key idea would involve employing a localized linearization of the drift function using an analogue of the mean value theorem. This linear approximation would then enable the application of advanced probabilistic techniques, such as the general chaining method, to effectively control the supremum of the associated stochastic process. The general chaining approach is particularly well-suited to navigating the complexities of nonlinearity, as it provides tight bounds for stochastic processes over high-dimensional parameter spaces.

However, addressing the nonlinear case introduces substantial technical challenges. These include the increased complexity of the Karush-Kuhn-Tucker conditions and the need for more stringent regularity assumptions on the drift function. As a result, this paper focuses on the linear case, where the theoretical framework is more tractable, and precise insights can be obtained without the additional burdens introduced by nonlinearity.

Nonetheless, exploring the nonlinear case represents a promising avenue for future research. Extending these methods to nonlinear drift structures would broaden their applicability, enabling the analysis of diffusion processes with more intricate and realistic dynamics.
\qed
\end{remark}

\section{Examples of pre-estimators}\label{s: preestimators}
In the adaptive Lasso framework, the weight vector is generally determined by employing a pre-estimation of the parameter. This step is essential for enhancing the statistical performance of simultaneous inference and variable selection. The pre-estimation provides an initial approximation of the coefficients, effectively serving as a ``first guess".

In this section, we present examples of pre-estimators tailored to different scenarios involving model parameters and their dimensions. Specifically, we consider two distinct cases.
In the first case, we assume that the expected value of the entire empirical covariance matrix of the diffusion process, $C_\infty$, is positive definite. Under this assumption, the standard Lasso estimator can be effectively employed as the pre-estimator.
In the second case, we address scenarios where the number of parameters $p$ is significantly larger than the dimension $d$ of the process or the duration of observation $T$. For such high-dimensional settings, we propose an appropriate pre-estimator that leverages the assumption of partial orthogonality.
These two approaches highlight the flexibility of the adaptive Lasso method in accommodating a variety of model characteristics, ensuring robust and reliable parameter estimation.

\subsection{Lasso estimator} \label{sec3.1}
%\hl{This will go in an example where we can use lasso as the pre-estimator, i.e. the case where $l_{min}(C_\infty)>0$. Explain that this condition implies indirect relation between $p$ and $d$, so $p$ cannot be very large, unless $d$ is. }\\
As is common in the literature on diffusion processes, we first consider the case where $C_\infty$ (as defined in \eqref{Cinfty}) is positive definite. In this setting, the Lasso estimator for the drift parameter of a diffusion process is defined as
\begin{equation}
    \widehat{\theta} := \argmin{\theta \in \Theta} \{L_T(\theta) + \lambda_l\|\theta\|_1\},
\end{equation}
where $\lambda_l$ is a tuning parameter.

Numerous studies in the literature have investigated the performance of the Lasso estimator in high-dimensional diffusion processes. For instance, in \cite{Pod22}, under the assumption $l_\text{min}(C_\infty) > 0$ and with an appropriately chosen $\lambda_l$, it is demonstrated that for drift functions of the form \eqref{linearModel}:
\begin{equation}\label{lassoConsistency} 
\|\theta_{0} - \widehat{\theta}\|_{\infty} \leq \| \theta_0 - \widehat{\theta} \|_2 = O_{\P}\left( \sqrt{\frac{s \log(p)}{dT}} \right). 
\end{equation}
If we define
\begin{equation}\label{ratio} 
r_T: = \sqrt{\frac{dT}{s\log(p)}}, 
\end{equation}
then in scenarios where this assumption holds, the Lasso estimator is relevant as a pre-estimator for the adaptive Lasso defined in \eqref{definitiAdaptiveLasso}. It is $r_T$-consistent for $\theta_0$, and the proxy $\eta_0$ is identical to $\theta_0$. Consequently, 
 Assumption \hyperref[assumption:B1]{\((\mathcal{B}_1)\)} holds with $\mathcal{M}_{2,T}=0$,  
 and, as noted in Remark \ref{remarkAdIC}, the adaptive irrepresentable condition in Assumption
 \hyperref[assumption:B2]{\((\mathcal{B}_2)\)} is automatically satisfied. 

 However, assuming $l_\text{min}(C_\infty) > 0$ introduces a restriction on the relationship between the model parameter dimensions. Specifically, this assumption implies that the dimension of the true parameter, $p$, must satisfy $p \leq d$, as the rank of $C_\infty$ is $\min\{p, d\}$. While this assumption is standard, it becomes restrictive when generalizing results to a high-dimensional framework where $p \gg d$ is of interest.

Although some results on the Lasso estimator for the drift parameter have been obtained without the requirement $l_\text{min}(C_\infty) > 0$ for linear drift functions (e.g., \cite{pina24}), alternative technical assumptions—albeit less restrictive—are necessary. These assumptions can often be challenging to implement or verify in practice.

To address this limitation, in the next subsection, we introduce a marginal estimator designed to serve as a suitable pre-estimator in scenarios where a partial orthogonality assumption is made. This approach allows us to extend our results to a genuinely high-dimensional framework, accommodating cases where $p \gg d$.

\subsection{Pre-estimation under the partial orthogonality assumption}\label{ss: partialOrtho}

When $p \gg d$, the assumption from the previous subsection becomes infeasible. To address this limitation, one potential strategy is to introduce an analogue of the restricted eigenvalue condition tailored to the diffusion process context. However, in this section, we adopt an alternative approach that ensures the necessary assumptions are satisfied under a partial orthogonality condition, as described below.

We consider the marginal estimator defined by: 
\begin{equation}\label{ortest} 
\widetilde{\theta} := -\frac{1}{T}\int_0^T \Phi^{\star}(X_t) dX_t, 
\end{equation} 
which can be viewed as an estimator of the quantity $\eta_0 := C_\infty \theta_0$.

\begin{remark}\rm 
It is possible to generalize this concept to a bridge-type estimator of the form: 
\begin{equation} 
\widetilde{\theta}^j := \left| \frac{1}{T}\int_0^T (\Phi^{\star}(X_t))_j dX_t\right|^\gamma, 
\end{equation} 
with certain $\gamma > 0$. However, for simplicity, we concentrate on the simpler case given by \eqref{ortest}. \qed 
\end{remark}

\noindent
Consider the following assumptions: \\

%(Partial orthogonality) Assume that covariates with zero coefficients coefficients and covariates  with nonzero coefficients are weakly correlates, i.e. $|C_\infty^{i, j}|\leq \rho_T$ for all $i \in \supp$ and $j\notin\supp$, and for a certain $\kappa\in[0,1]$, the following inequality holds 

\noindent \textbf{Assumption ($\mathcal{D}_1$)}\label{assumption:D1} \textbf{(Partial Orthogonality):} Assume that the covariates associated with zero coefficients and those associated with nonzero coefficients are weakly correlated: 
\[
|C_\infty^{i, j}| \leq \rho_T, \quad \forall i \in \supp \text{ and } j \notin \supp,
\]
where $\rho_T$ is a small positive value. Furthermore, for a certain $\kappa \in [0,1]$, the following inequality holds:

\begin{equation}\label{ctort}
    \mathfrak{c}_T := \left(  \max_{j\notin \supp}|\eta_0^j|\right) \left( \sum_{j\in\supp} \frac{1}{s|\eta_0^j|^{2}} \right)^{\frac{1}{2}} \leq \frac{\kappa \tau_{\text{min}}}{s\rho_T}.  
\end{equation}

\noindent \textbf{Assumption ($\mathcal{D}_2$):}\label{assumption:D2} The minimum $\eta_{0,\supp}^\text{min} := \min{\{ |\eta_0^j |, j\in \supp \}}$ satisfies 

\begin{equation}
    \frac{\sqrt{s}(1 + \mathfrak{c}_T)}{\eta_{0,\supp}^\text{min} r_T}\to 0 \quad \text{and} \quad  \frac{r_T^2\log{p}}{T}\left( \mathcal{K}s^2 (\theta_{0,\supp}^\text{max})^2 + \sqrt{\mathcal{K}} + \mathfrak{M}\right)\to 0. 
\end{equation}

\noindent
The following theorem establishes the validity of using the marginal estimator defined in \eqref{ortest} as the initial estimator for the adaptive Lasso, provided the partial orthogonality assumption holds.

\begin{theorem}\label{Theorem3}
    Assume~\hyperref[assumption:A1]{\((\mathcal{A}_1)\)}-\hyperref[assumption:A3]{\((\mathcal{A}_3)\)}, \hyperref[assumption:B3]{\((\mathcal{B}_3)\)}, \hyperref[assumption:C]{\((\mathcal{C})\)} and \hyperref[assumption:D1]{\((\mathcal{D}_1)\)}-\hyperref[assumption:D2]{\((\mathcal{D}_2)\)}  hold. Then the marginal estimator $\widetilde{\theta}$ defined in \eqref{ortest} satisfies \hyperref[assumption:B1]{\((\mathcal{B}_1)\)}, i.e. it is $r_T$-consistent for $\eta_0$, and the adaptive irrepresentable condition in \hyperref[assumption:B2]{\((\mathcal{B}_2)\)} holds with $\kappa$ from \hyperref[assumption:D1]{\((\mathcal{D}_1)\)}. 
\end{theorem}

\begin{proof}
Recalling that $\eta_0 = C_{\infty}\theta_0$, by definition of the marginal estimator, for all $j=1,\dots,p$ we have that 
\begin{gather}
    \widetilde{\theta}^j = \sum_{i=1}^s \frac{\theta_0^i}{T}\int_0^T \phi_i(X_t)^{\star}\phi_j(X_t)dt - \frac{1}{T} \int_0^T \phi_j(X_t)^{\star} dW_t\\
    =\eta_0^{j} + \sum_{i=1}^s\theta_0^i \left( C_T^{ij} - C_{\infty}^{ij} \right) - \epsilon_T^j.
\end{gather}
First, we prove the $r_T$-consistency for $\eta_0$. For any $\varepsilon>0$,
\begin{gather}
    \P \left( r_T\|\widetilde{\theta}-\eta_0\|_{\infty}>\varepsilon \right) \leq 
    \P \left( r_T\max_{1\leq j \leq p}|\sum_{i=1}^s\theta_0^i \left( C_T^{ij} - C_{\infty}^{ij} \right)|>\varepsilon/2 \right)\\
    +\P \left( r_T\max_{1\leq j \leq p}|\epsilon_T^j|>\varepsilon/2 \right)
    \leq p \max_{1\leq j \leq p}\P \left( r_T|\sum_{i=1}^s\theta_0^i \left( C_T^{ij} - C_{\infty}^{ij} \right)|>\varepsilon/2 \right)\\
    +p \max_{1\leq j \leq p} \P \left( r_T |\epsilon_T^j|>\varepsilon/2 \right)
    \leq p \max_{1\leq i, j \leq p}\P \left( \left| C_T^{ij} - C_{\infty}^{ij} \right|> \frac{\varepsilon}{2s\theta_{0,\supp}^\text{max}r_T} \right) \\
    +
    p \max_{1\leq j \leq p} \P \left( |\epsilon_T^j|>\frac{\varepsilon}{2r_T} \right)\\
    \leq 6p\exp\left(\frac{-T\varepsilon^2}{144\mathcal{K}s^2 (\theta_{0,\supp}^\text{max})^2r_T^2}\right)
    +2p \exp\left(\frac{-T\varepsilon^2}{8(\sqrt{\mathcal{K}}+\mathfrak{M})r_T^2}\right) 
    + 6p\exp\left(-\frac{T}{36}\right), 
\end{gather}
where the last inequality holds due to the Lemma \ref{boundOperatorCovarianceMatrices} and Proposition \ref{propMTGbernstein}, and all three terms converge to zero under the Assumption \hyperref[assumption:D2]{\((\mathcal{D}_2)\)}. 

Now, for $\mathcal{M}_{2,T}=\max_{j\notin\supp}|\eta_0^j|$, according to the the Assumption \hyperref[assumption:D2]{\((\mathcal{D}_2)\)}, we get %(\textcolor{red}{maybe here the exponent 2 must be inside?})
\begin{equation}
    \left( \sum_{j\in \supp}\left( \frac{1}{|\eta_0^j|} + \frac{\mathcal{M}_{2,T}}{|\eta_0^j|^2}\right)\right)^2 \leq \frac{2s}{(\eta_{0,\supp}^\text{min})^2}(1 + \mathfrak{c}_T^2)=o(r_T^2) 
\end{equation}
 and \hyperref[assumption:B1]{\((\mathcal{B}_1)\)} holds. 
 
 To show that the adaptive irrepresentable condition \hyperref[assumption:B2]{\((\mathcal{B}_2)\)} holds, we first notice that for all $j\notin\supp,$ Assumption \hyperref[assumption:D1]{\((\mathcal{D}_1)\)} gives the bound $\|( C_\infty^{\supp^C \supp} )_{j-s}\|_2^2\leq s \rho_T^2 $ and 

\begin{equation}
|\eta_0^j|\left| \left( C_\infty^{\supp^C \supp} (C_\infty^{\supp\supp })^{-1} u \right) ^{j-s} \right|\leq \left\|\left( C_\infty^{\supp^C \supp} \right)_{j-s}\right\|_2 \left\| 
 (C_\infty^{\supp\supp })^{-1}\right\|_\text{op} |\eta_0^j|\| u\|_2 \leq \frac{s\rho_T \mathfrak{c}_T}{\tau_\text{min}}\leq \kappa,  
\end{equation}
where the last inequality holds due to \eqref{ctort} and finishes the proof. 
\end{proof}

\begin{remark}\rm 
A special case of Theorem \ref{Theorem3} arises when $\rho_T = o(T^{-1/2})$, implying that the covariates with nonzero and zero coefficients are essentially uncorrelated. In this scenario, assuming that $\theta_{0,\supp}^\text{max} + \mathcal{K} + \mathfrak{M} = O(1)$, we can set $\eta_0^j = 0$ for $j \notin \supp$, and condition \eqref{ctort} is satisfied. As a result, the marginal estimator estimator $\widetilde{\theta}$ in \eqref{ortest} is sign consistent with the rate $r_T=o(\sqrt{T/(s^2\log(p))})$ and the adaptive irrepresentable condition \hyperref[assumption:B2]{\((\mathcal{B}_2)\)} is automatically satisfied. \qed  
\end{remark}

\section{Concentration inequalities}\label{s: conineqsection}
%In order to prove the support recovery property for the adaptive Lasso, it is essential to establish a control over several probability sets that involve an additive functional structure of the form 
%\begin{equation}
 %   \frac{1}{T}\int_0^T f(X_t)dt,
%\end{equation}
%where $f:\R^d \rightarrow \R$ and $X_t$ is the diffusion process defined in \eqref{model}. One of the most popular techniques used to address these types of problems is the so-called concentration inequalities.
%Concentration inequalities (CIs) for additive functionals play a crucial role in fields such as statistics, machine learning, and mathematical finance. Their popularity has grown as they provide a sophisticated tool to quantify the deviation of an estimator or a stochastic process from its expected value or a target quantity. 

To establish the support recovery property for the adaptive Lasso in the context of diffusion processes, it is essential to effectively control probabilities associated with additive functionals of the form: 
\begin{equation}
    \frac{1}{T}\int_0^T f(X_t)dt,
\end{equation}
where $f: \R^d \to \R$ is a measurable function, and $X_t$ is the diffusion process defined in \eqref{model}.

The control of these additive functionals is achieved using concentration inequalities that account for the probabilistic behavior of  $X_t$ and the growth conditions on $f$. These inequalities rely on a careful interplay between assumptions about the properties of $f$ 
and the ergodicity and regularity of the diffusion process $X$. 
For example, $f$ is usually assumed to belong to a class of functions defined by parameters that control its growth rate (e.g., polynomially bounded growth) and a scaling factor. Additionally, the diffusion process 
$X$ must satisfy certain ergodicity conditions, such as subexponential or exponential mixing rates, which depend on its drift and diffusion coefficients. These concentration inequalities provide explicit bounds on the likelihood that the additive functional deviates from its mean. Specifically, they ensure that for any $\varepsilon>0$ and sufficiently large $T$, the probability of the deviation
\[
\left| \frac{1}{T}\int_0^T f(X_t)dt - \mu(f) \right| > \varepsilon
\]
is exponentially small. Here $\mu(f) := \int_{\R^d} f(x) \mu(dx)$ is the mean of $f$ with respect to the invariant measure $\mu$ of $X_t$. 

Such results are typically derived using mathematical tools like the Poisson equation, martingale approximations, and the mixing properties of $X$. These approaches ensure that the derived bounds explicitly account for the dimensionality of the model and align with the structural properties of both the diffusion process and the function $f$.
These inequalities are pivotal in proving that the adaptive Lasso accurately identifies the true support of the parameter in high-dimensional diffusion settings.

%Several results in the field of concentration inequalities for additive functionals in the context of diffusion processes have been presented in \cite{pina24, Str21, dje04, Tro23, var19} under specific assumptions of the process $X_t$ or the own function $f$. To illustrate such techniques,  we will use the following concentration inequality from Theorem 3 and Section 3 in \cite{Sauss12}. 

Several key results regarding concentration inequalities for additive functionals in diffusion processes are discussed in works such as \cite{Str21, pina24, dje04, Tro23, var19}. These results are established under particular assumptions about the diffusion process $X$ and the function $f$. As an example of how such techniques are applied, consider the following concentration inequality derived from Theorem 3 and Section 3 in \cite{Sauss12}.

\begin{theorem}\label{concentrationInequalitytheo}
Assume that ~\hyperref[assumption:A1]{\((\mathcal{A}_1)\)}-\hyperref[assumption:A3]{\((\mathcal{A}_3)\)} hold and let $f 
    : \R^d \rightarrow \R$ be a Lipschitz function. Then there exists a constant $K>0$, independent of $d$, such that for all $\mu \in \R$,
    \begin{equation}
        \E\left[\exp\left(\frac{\mu}{T}\int_0^T\left(f(X_t) - \E\left[f(X_t)\right]\right)dt\right)\right] \leq \exp\left(K \mu^2 \nlip{f}^2/T\right).
    \end{equation}
\end{theorem}

\noindent
This result provides explicit bounds on the deviations of additive functionals, demonstrating an example of conditions—such as Lipschitz continuity—under which these deviations can be controlled with high probability. 

In particular, the following proposition, which is a direct consequence of Theorem \ref{concentrationInequalitytheo}, provides a method to determine the constant $\mathcal{K}$ in Assumption \hyperref[assumption:C]{\((\mathcal{C})\)} under the Lipschitz continuity condition.

\begin{proposition}\label{lipcondC}
    Assume that ~\hyperref[assumption:A1]{\((\mathcal{A}_1)\)}-\hyperref[assumption:A3]{\((\mathcal{A}_3)\)} hold and for all $i,j=1,\dots,p$ the functions $\left\langle \phi_i , \phi_j\right\rangle:\R^d \rightarrow \R$ are assumed to be Lipschitz continuous, with Lipschitz constants uniformly bounded in $p, s$ and $d$. Then, Assumption \hyperref[assumption:C]{\((\mathcal{C})\)} holds with $\mathcal{K}=\nop{\mathcal{Q}}$, where the  matrix $\mathcal{Q} \in \M{p}{p}$ and $\mathcal{Q}_{ij}:=\|\left\langle \phi_i , \phi_j\right\rangle\|_{\text{\rm Lip}}$.
\end{proposition}

\begin{remark}\rm 
Assumption \hyperref[assumption:C]{\((\mathcal{C})\)}, 
which ensures sub-Gaussian concentration inequalities for the martingale term, can be relaxed to the alternative Condition 
\(\mathcal{C}'\) below,
allowing for sub-exponential growth:

\smallbreak
\noindent \textbf{Assumption ($\mathcal{C}'$):}\label{assumption:Cprime}  For the random functional $H_v(X_{[0,T]})$ defined in \eqref{Hvdef} there exist constants $\mathfrak{C}$ and $1\leq \mathfrak{d} \leq 2$ such that that for all $x>0$  and for all $v\in \R^p$ satisfying $\|v\|_2 = 1$
   % \begin{equation}\label{conineq}
    %    \E\left[\exp\left(\mu H_v(X_{[0,T]}) \right)\right] \leq \exp\left(\mathcal{K} \mu^\mathfrak{d} /T\right).
    %\end{equation}

    \begin{equation}
    \P\left(H_v(X_{[0,T]}) \geq x  \right) \leq 
    \exp\left(-\mathfrak{C}Tx^\mathfrak{d}\right). 
\end{equation}
This relaxation broadens the applicability of the results to settings where sub-Gaussian assumptions may be too restrictive, while still ensuring meaningful concentration bounds.

Adapting the proofs of the main results under Assumption \hyperref[assumption:Cprime]{\((\mathcal{C}')\)} involves only minor modifications. Specifically, straightforward adjustments to Lemma \ref{boundOperatorCovarianceMatrices}, Proposition \ref{propMTGbernstein}, and subsequent results suffice to incorporate the weaker growth constraints introduced by Assumption \hyperref[assumption:Cprime]{\((\mathcal{C}')\)}. For clarity and brevity, detailed derivations are omitted here.

However, it is important to note that relaxing the assumption to \hyperref[assumption:Cprime]{\((\mathcal{C}')\)} imposes a trade-off. Specifically, the maximum permissible number of non-zero parameters in the unknown vector $\theta_0$
decreases compared to the bounds established under Assumption \hyperref[assumption:B3]{\((\mathcal{B}_3)\)} and detailed in Remark \ref{r: ConclusionParameters}. This reduction reflects the balance between relaxing the martingale term's assumptions and tightening the sparsity constraints on 
$\theta_0$. Despite these limitations, the relaxed Assumption \hyperref[assumption:Cprime]{\((\mathcal{C}')\)} remains practically significant for high-dimensional models where sub-Gaussian assumptions may be overly stringent, offering flexibility for a broader range of applications.
\qed  
\end{remark}

\noindent
Employing the relaxed Assumption \hyperref[assumption:Cprime]{\((\mathcal{C}')\)}
opens the door to alternative forms of concentration inequalities beyond the sub-Gaussian regime. Notably, this includes the use of classical log-Sobolev inequalities, which provide a powerful framework for deriving sub-exponential concentration bounds. Log-Sobolev inequalities are particularly effective for capturing the concentration properties of processes with weaker regularity or heavier tails, making them a natural choice under the broader scope of 
Assumption \hyperref[assumption:Cprime]{\((\mathcal{C}')\)}.

For Ornstein-Uhlenbeck (OU) type processes, an alternative route to concentration inequalities involves Malliavin calculus techniques. These rely on the properties of the Malliavin derivative, which satisfies the classical chain rule and offers a rigorous framework for studying the regularity of functionals of the OU-process. As developed in \cite{NV09} and applied in \cite{Cio20}, this approach is especially well-suited for establishing sub-exponential concentration inequalities.

The Malliavin calculus framework provides several advantages for analyzing the OU-process. First, it exploits the smoothness of the probability density associated with the OU-process, enabling the use of integration by parts formulas that are central to deriving sharp bounds. Second, it offers precise control over the deviations of functionals by linking the stochastic dynamics of the process with the geometry of its underlying probability space. This connection facilitates the derivation of sub-exponential concentration inequalities under conditions that are less stringent than those typically required for sub-Gaussian behavior.

Another effective method for establishing sub-exponential (or even stronger) concentration inequalities involves the Poisson equation, as explored in \cite{Tro23}. This approach uses solutions to the Poisson equation to directly manage deviations of additive functionals. It is particularly effective for a wide range of diffusion processes and provides a means to slightly relax and extend Assumptions ~\hyperref[assumption:A1]{\((\mathcal{A}_1)\)}-\hyperref[assumption:A3]{\((\mathcal{A}_3)\)}. By doing so, it broadens the applicability of the theoretical results to scenarios with more general dynamics or weaker regularity conditions.

However, deriving and interpreting the explicit constant \(\mathfrak{C}\)
 in this framework poses significant challenges. The difficulty stems from the dependence of 
\(\mathfrak{C}\) on the solution of the Poisson equation, which is often intricately tied to properties of the underlying diffusion process, such as its invariant measure and ergodic rates. Determining 
\(\mathfrak{C}\) typically requires additional analytical effort, including a detailed examination of the growth and regularity of the drift and diffusion coefficients. Despite these challenges, the Poisson equation approach remains a powerful and versatile tool. It provides a pathway to sharper concentration bounds and offers greater flexibility in modeling assumptions, making it a valuable asset in the study of high-dimensional diffusion processes.

%{\color{red} Mark: Here we should also discuss alternative concentration inequalities. This may include classical log-Sobolev inequalities, Malliavin-based results for the quadratic case [very important for OU-processes] and recent results by Claudia. }

\section{Numerical studies}\label{s: numericalStudies}

This section presents several numerical experiments conducted on simulated data to validate and demonstrate our theoretical results. Specifically, we evaluate the performance of the adaptive Lasso estimator in comparison to the classical maximum likelihood estimator (MLE) and the standard Lasso estimator in estimating the drift function's parameters within a high-dimensional diffusion process.

As described in Section \ref{s: preestimators}, the adaptive Lasso computation requires an initial pre-estimator. For these experiments, the standard Lasso estimator was employed as the pre-estimator, although other alternatives could be used depending on the context. The experiments include four main evaluations: 
(a) A heatmap diagram is utilized to visually compare the prediction accuracy of the adaptive Lasso, MLE, and standard Lasso estimators, (b) A graphical analysis is presented to demonstrate the variable selection capabilities of each estimator, (c) The $\ell_1$ and $\ell_2$ error norms are analyzed as the dimensionality of the model parameters increases, providing insight into the prediction accuracy under different high-dimensional settings and (d) We evaluate the asymptotic normality of the adaptive Lasso estimator across different time horizons.

Inspired by the numerical study in \cite{pina24} and using the YUIMA package \cite{yuima}, we simulate a 
$d$-dimensional diffusion process as described in \eqref{model}, where the linear drift function is defined by:
\begin{equation}\label{simulatedProcess}
    \bto(x) = 3sx + \sum_{i=1}^p \theta_0^i \cos\left((i+1)x\right),
\end{equation}
where the underlying parameter of interest  $\theta_0 \in \R^p$ (here $\cos(y)$ is computed componentwise for $y\in \R^d$). Furthermore, the value of each component $\theta^i$ is properly restricted to ensure that the drift function satisfies  Assumptions \hyperref[assumption:A1]{\((\mathcal{A}_1)\)}-\hyperref[assumption:A3]{\((\mathcal{A}_3)\)}. Additionally, Assumption \hyperref[assumption:C]{\((\mathcal{C})\)} is verified by Proposition \ref{lipcondC}. Finally, Assumptions \hyperref[assumption:B1]{\((\mathcal{B}_1)\)}-\hyperref[assumption:B3]{\((\mathcal{B}_3)\)} are satisfied due to the use of the Lasso estimator as the pre-estimator.

Figure \ref{fig:heatmap} illustrates a heatmap comparison of the performance of the three estimators, highlighting the structural properties of the parameter vector of interest. The heatmap visually represents the sparsity of the parameter: blank spaces correspond to zero components, while shaded squares indicate nonzero components, with the shading intensity reflecting the magnitude of the respective values.

The analysis is based on simulations of a 5-dimensional diffusion process structured as \eqref{simulatedProcess}, with a drift parameter dimensionality of $p=30$ and an observation period of $T=10$. 
This controlled setup evaluates each estimator's capability to identify both the sparsity pattern and the magnitude of the parameter vector effectively.

While the theoretical results assume continuous observation of the process trajectories, practical computational constraints necessitate a discretized approximation for simulation purposes. For this study, an equidistant discretization step of $\Delta_n = 0.05$
was employed, as tests indicate that further reductions in the step size yield negligible accuracy improvements. This observation aligns with findings in \cite{Gai19}, which offers an in-depth analysis of the impact of discretization steps on high-dimensional diffusion process simulations.
The tuning parameters for both the Lasso and adaptive Lasso estimators were determined through cross-validation, in line with the methodology described in \cite{Cio20}.

\begin{figure}[H] 
    \centering
    \includegraphics[width=0.9\textwidth, height=5cm]{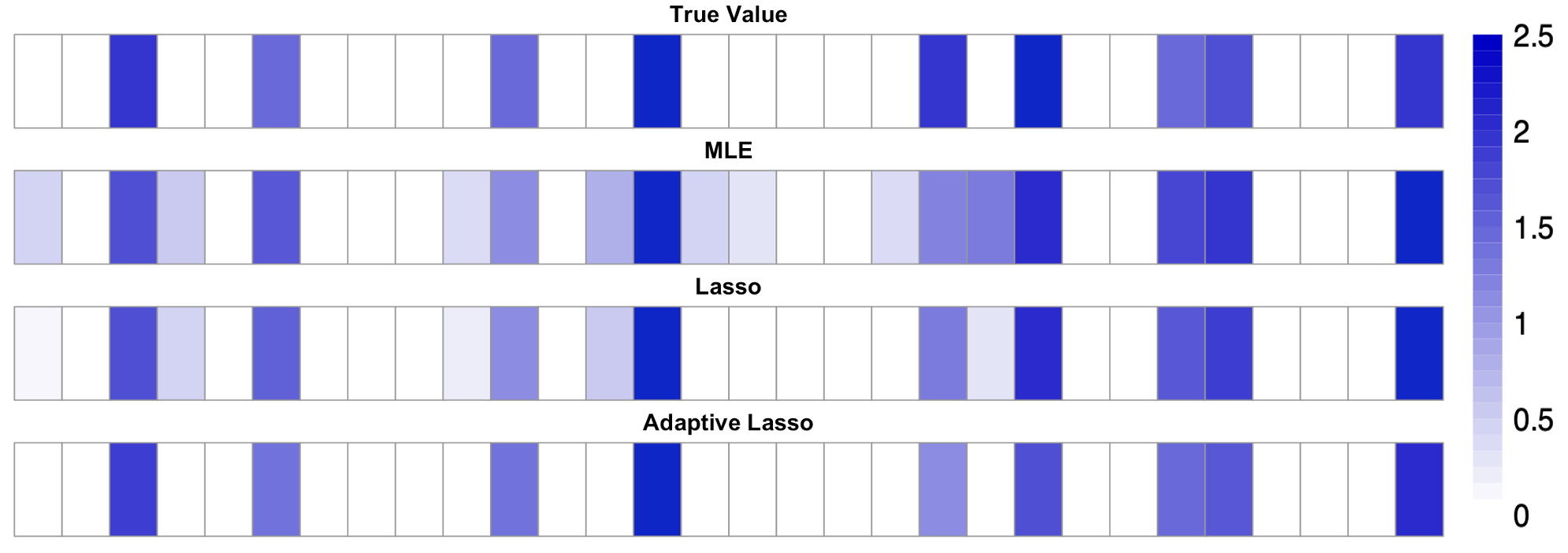} % Ajusta el ancho según sea necesario.
    \captionsetup{margin=1cm} % Ajusta los márgenes del pie de foto.
    \caption{Comparison of the overall performance of the MLE, Lasso estimator, and adaptive Lasso in terms of support recovery and prediction accuracy.}
    \label{fig:heatmap} % Etiqueta para referenciar la figura en el texto.
\end{figure}

 \noindent 
 Figure \ref{fig:heatmap} reveals that the Lasso and its adaptive variant substantially outperform the classical MLE in support recovery. The classical MLE struggles to capture the sparsity structure of the true parameter, even for relatively low-dimensional settings of $d$ and $p$.
In contrast, the adaptive Lasso demonstrates a marked improvement over the standard Lasso, providing superior accuracy in identifying the support of the parameter vector.

These findings are further corroborated by Figure \ref{fig:SuppRecEvaluation}, which reinforces the adaptive Lasso's advantage in support recovery performance. Both figures highlight the adaptive Lasso's effectiveness in balancing accurate variable selection with robust parameter estimation in high-dimensional diffusion processes.
 
 \begin{figure}[H] 
    \centering
    \includegraphics[width=0.87\textwidth, height=7.5cm]{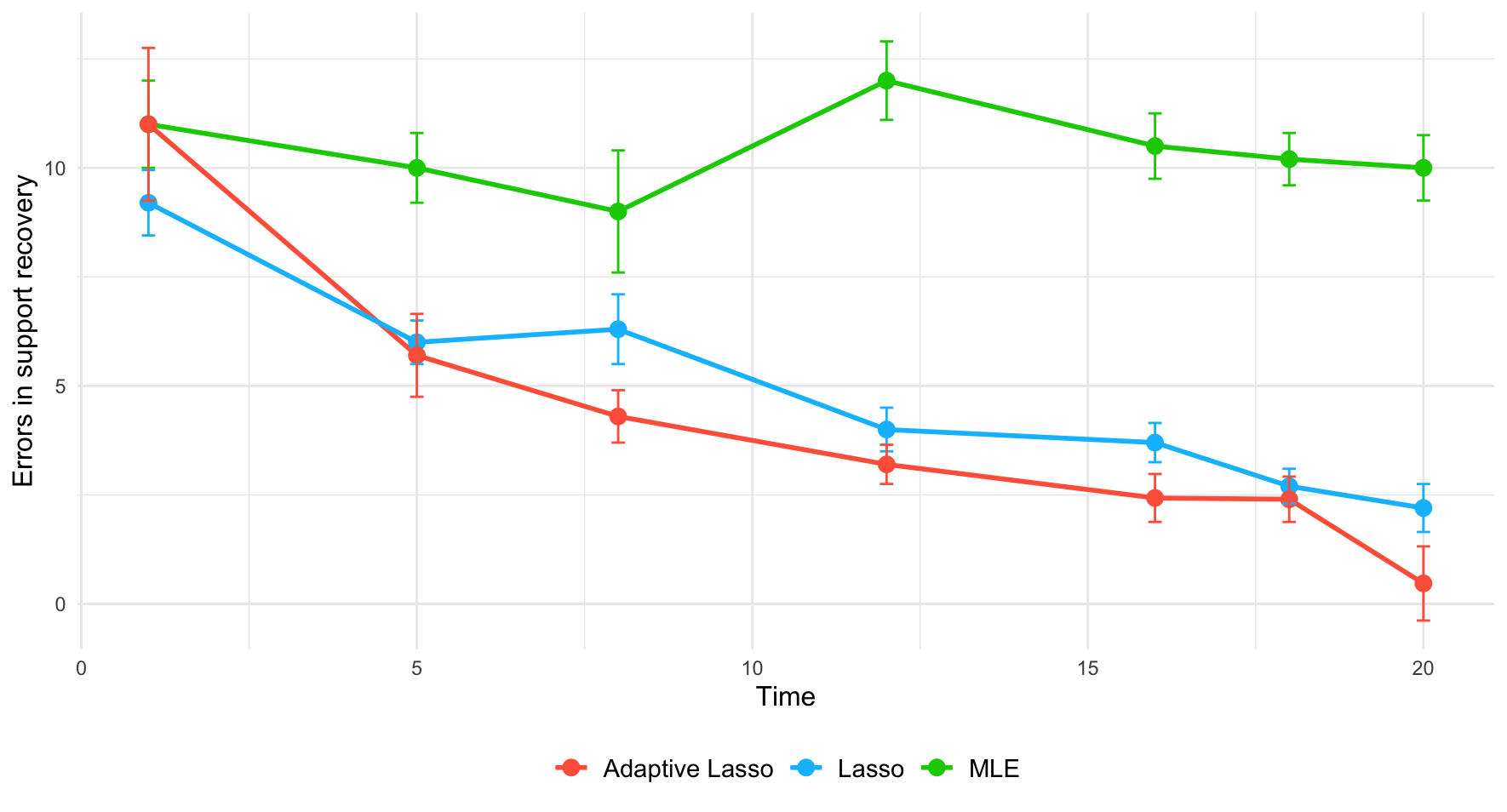} % Ajusta el ancho según sea necesario.
    \captionsetup{margin=1cm} % Ajusta los márgenes del pie de foto.
    \caption{Comparison of the support recovery errors over time for the MLE, Lasso estimator, and adaptive Lasso, including the corresponding standard deviations.}
    \label{fig:SuppRecEvaluation} % Etiqueta para referenciar la figura en el texto.
\end{figure}

In Figure \ref{fig:SuppRecEvaluation}, we evaluate the support recovery performance of the estimators by measuring the number of errors in identifying the parameter's support. A unit error corresponds to the misidentification of any component of the parameter vector. For this analysis, the process dimension is fixed at $d=5$, and the drift parameter dimension is set at $p=30$.
The estimators are evaluated over varying observation times across 25 iterations.

The results indicate that the MLE consistently exhibits a higher error rate, showing minimal improvement in support recovery as the observation time increases. Conversely, the Lasso-based techniques demonstrate significantly better performance, with the adaptive Lasso achieving near-perfect support recovery for sufficiently large observation windows. This underscores the adaptive Lasso's robustness and accuracy in identifying the true parameter's support, corroborating our theoretical findings.

Figures \ref{fig:L1error} and \ref{fig:L2error} further explore the performance of the estimators by comparing their $l_1$ and $l_2$ errors as the parameter dimension $p$
increases. Adopting a methodology akin to \cite{pina24}, the process dimension is fixed at $d=5$, and the observation time is $T=10$. For each of 25 iterations, the true parameter is generated with a sparsity of 0.75, where the non-zero values are uniformly distributed between 2 and 3. The Lasso and adaptive Lasso tuning parameters are selected using cross-validation.

The results confirm that both the $l_1$ and $l_2$ errors
increase with the parameter dimension for all estimators. However, the adaptive Lasso consistently achieves lower errors compared to the standard Lasso and MLE, further demonstrating its superior accuracy and efficiency in high-dimensional settings. These findings validate the practical advantages of the adaptive Lasso, particularly in challenging high-dimensional inference tasks.

 \begin{figure}[H] 
    \centering
    \includegraphics[width=0.9\textwidth, height=7cm]{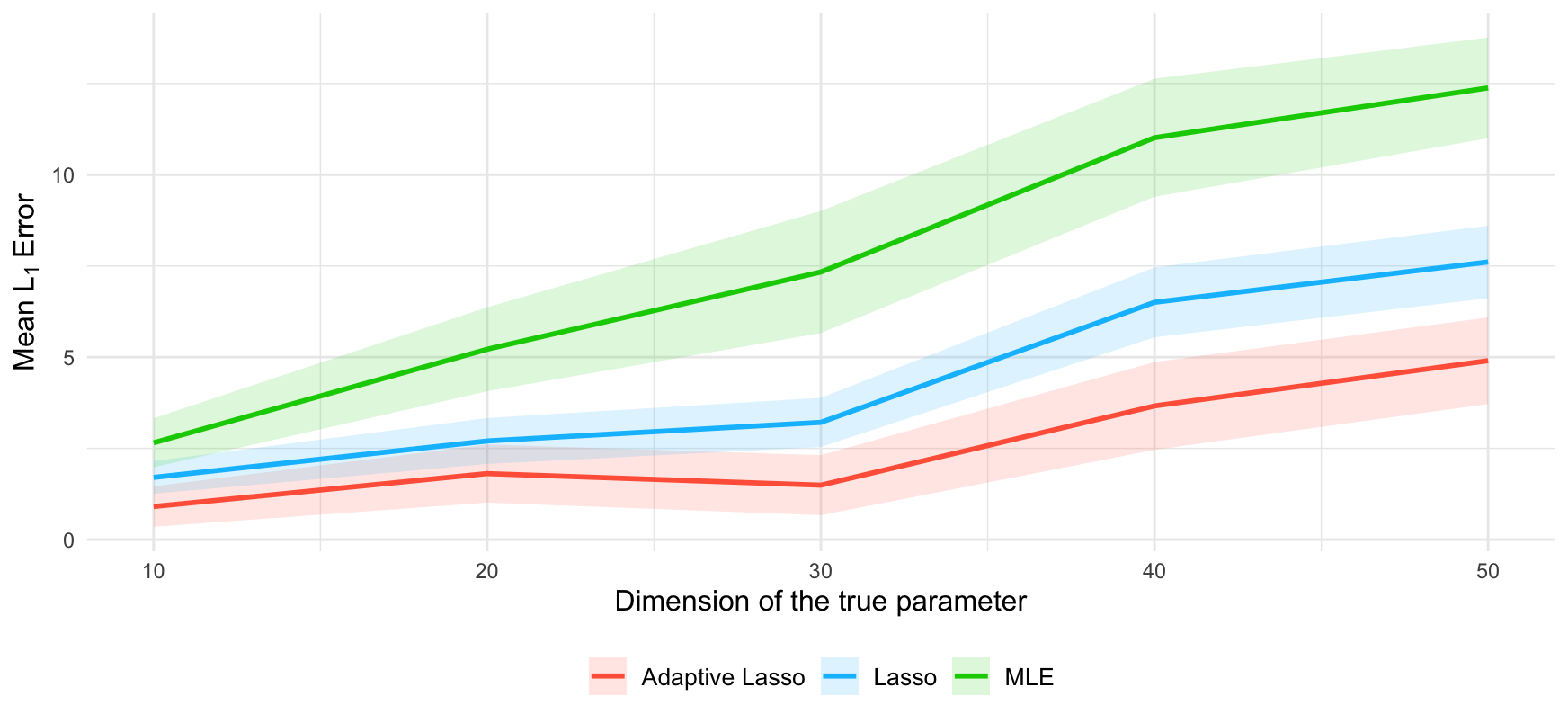} % Ajusta el ancho según sea necesario.
    \captionsetup{margin=1cm} % Ajusta los márgenes del pie de foto.
    \caption{$l_1$ mean error for the MLE, the Lasso estimator and the adaptive Lasso $\pm$ one standard deviation.}
    \label{fig:L1error} % Etiqueta para referenciar la figura en el texto.
\end{figure}

\begin{figure}[H] 
    \centering
    \includegraphics[width=0.9\textwidth, height=7cm]{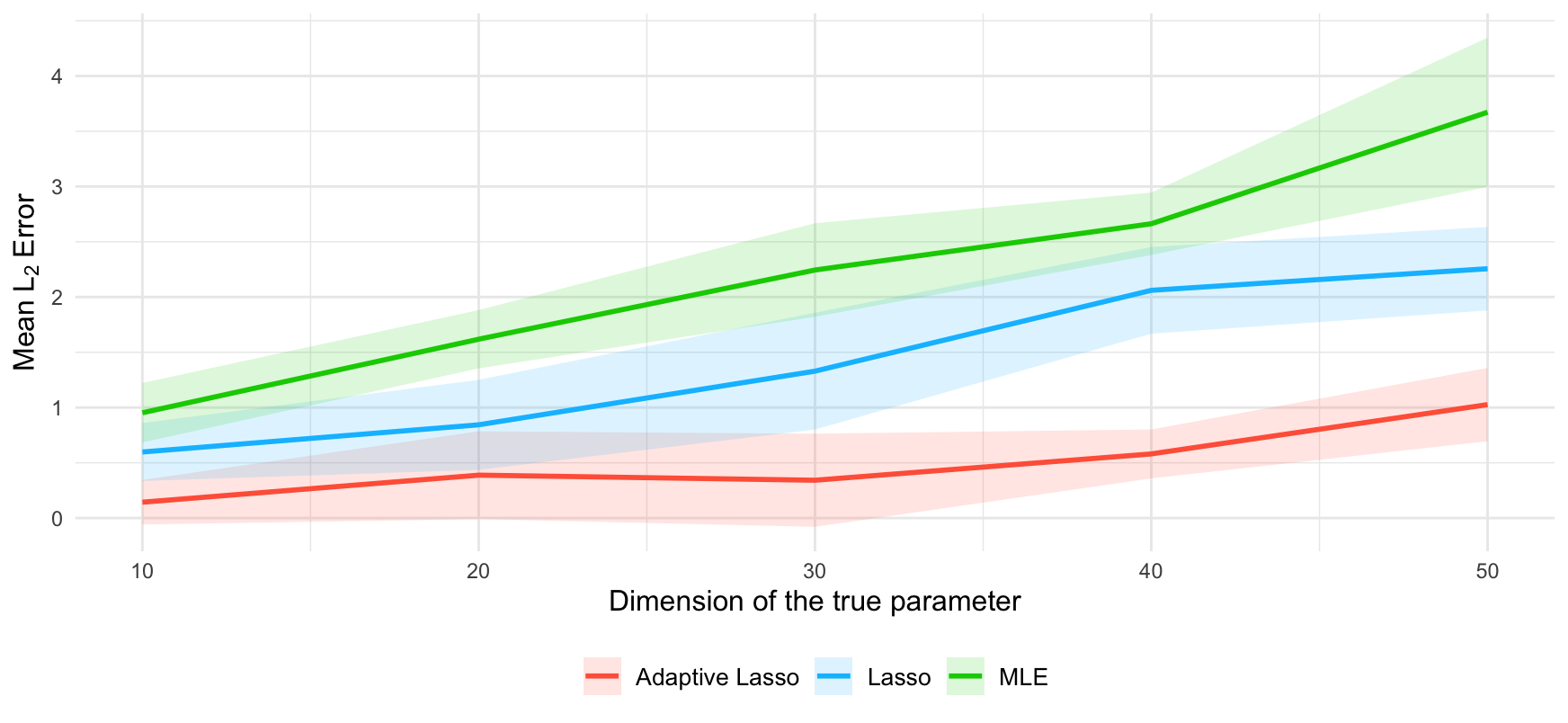} % Ajusta el ancho según sea necesario.
    \captionsetup{margin=1cm} % Ajusta los márgenes del pie de foto.
    \caption{$l_2$ mean error for the MLE, the Lasso estimator and the adaptive Lasso $\pm$ one standard deviation.}
    \label{fig:L2error} % Etiqueta para referenciar la figura en el texto.
\end{figure}

\noindent We observe that, for both metrics, the adaptive Lasso estimator exhibits a smaller error compared to the classical MLE and the standard Lasso estimator, particularly for larger values of $p$.

%Finally, the last experiment provides evidence of the asymptotic normality of the Adaptive Lasso estimator. To show this, as in the other scenarios, we simulate a 5-dimensional diffusion process of the form \eqref{simulatedProcess}, where the unknown parameter has dimension $p =30$ and is 80\% sparse. For each time horizon $T$, we generate 100 trajectories of the process. For each trajectory, we use adaptive Lasso to estimate the unknown parameter in order to compute the statistic defined in the first term of equation \eqref{astheoremstat}. Once we have computed the 100 statistics for each $T$, we plot the empirical density distribution as shown in Figure \ref{fig: AN}

Finally, the last experiment provides empirical evidence supporting the asymptotic normality of the adaptive Lasso estimator. To demonstrate this, we simulate a 5-dimensional diffusion process of the form \eqref{simulatedProcess}, where the unknown parameter has dimension $p = 30$ and exhibits 80\% sparsity. For each time horizon $T$, we generate 100 trajectories of the process. 

\begin{figure}[H]
    \centering
    \includegraphics[width=0.9\textwidth, height=8.3cm]{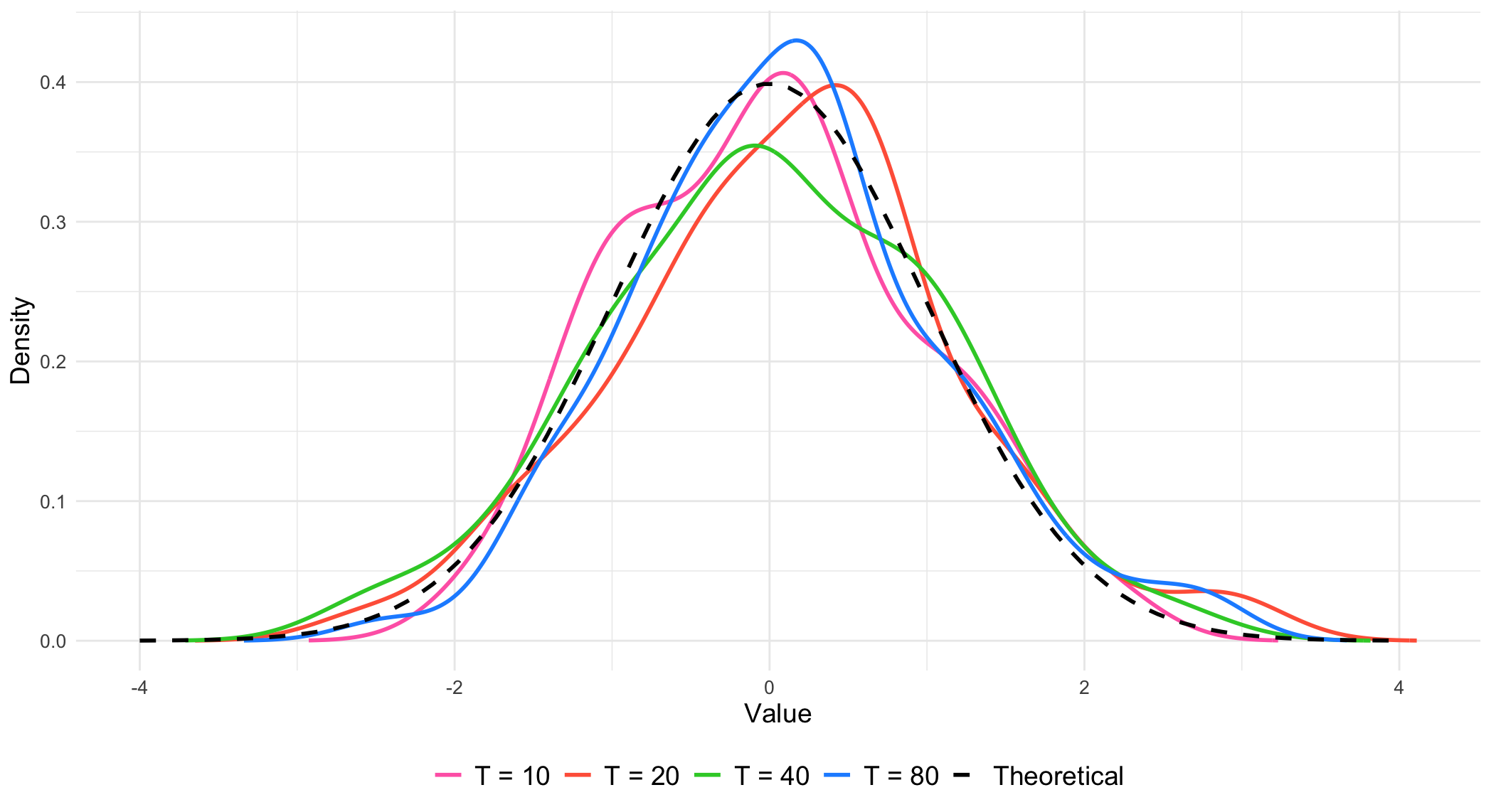} % Ajusta el ancho según sea necesario.
    \captionsetup{margin=1cm} % Ajusta los márgenes del pie de foto.
    \caption{Empirical density distributions  at different time horizons compared with the theoretical standard Gaussian density.
}\label{fig: AN}
\end{figure}

For each trajectory, we apply the adaptive Lasso to estimate the unknown parameter and compute the statistic defined by the first term of equation \eqref{astheoremstat} with a  vector $\alpha$ of equal weights. After obtaining the 100 computed statistics for each value of $T$, we plot their empirical density distributions. These plots, illustrated in Figure \ref{fig: AN}, provide a visual comparison of the convergence of the estimator's distribution to the standard normal distribution as \(T\) increases, confirming the theoretical asymptotic normality of the Adaptive Lasso estimator.

%We observe that as $T$ increases, the empirical density progressively converges towards the standard Gaussian distribution. This provides strong evidence that the estimator exhibits asymptotic normality, aligning with the theoretical results presented in Theorem \ref{theoremAsNor}.

We observe that as $T$ increases, the empirical density progressively converges towards the standard Gaussian distribution. This observation provides strong empirical evidence that the estimator exhibits asymptotic normality, in agreement with the theoretical results established in Theorem \ref{theoremAsNor}.

\section{Auxiliary results and proofs}\label{s: proofs}

% \subsection{A bound for the inverse of the covariance operator}

\subsection{Control of the Fisher information matrix via concentration inequalities}

To establish the support recovery property, we need to control both the elements of the Fisher information matrix $C_T$ and the operator norm of its significant part $C_T^{\supp\supp}$. Since we impose assumptions on the expected value of $C_T$, i.e. matrix $C_\infty$, we use concentration inequalities that directly follow from Assumption \hyperref[assumption:C]{\((\mathcal{C})\)}. These inequalities allow us to show that with high probability as $T\to\infty$ the matrices $C_T$ and $C_\infty$ become  close to each other in the following sense.

\begin{lemma}\label{boundOperatorCovarianceMatrices}
   Under Assumptions~\hyperref[assumption:A1]{\((\mathcal{A}_1)\)}-\hyperref[assumption:A3]{\((\mathcal{A}_3)\)} and \hyperref[assumption:C]{\((\mathcal{C})\)}, the inequalities  

\begin{equation}\label{uConDif}
    \sup_{u\in \R^p: \| u \|_2=1}\P\left( u^{\star}(C_T - C_{\infty})u \geq x\right) \leq \exp\left(-\frac{Tx^2}{4\mathcal{K}}\right)
\end{equation}
and \begin{equation}\label{ijConDif}
    \max_{1\leq i,j\leq p}\P\left( |C_T^{ij} - C_{\infty}^{ij}| \geq x\right) \leq 6\exp\left(-\frac{Tx^2}{36\mathcal{K}}\right)
\end{equation}
hold for $x>0$. 

\end{lemma}
\begin{proof}
 By applying Chebyshev's inequality under the given Assumption \hyperref[assumption:C]{\((\mathcal{C})\)}, we obtain that for all $\mu>0$ and $u\in \R^p$ satisfying $ \| u \|_2=1$, 
\begin{equation}
    \P\left(H_u(X_{[0,T]}) \geq x \right) \leq \E\left[\exp\left(\mu H_u(X_{[0,T]}) \right)\right]\exp(-\mu x)\leq 
    \exp\left(\mathcal{K} \mu^2 /T - \mu x\right). 
\end{equation}
If we now minimize this expression with respect to $\mu$, we obtain the first inequality \eqref{uConDif}.

To establish the second inequality, we can notice that

\begin{gather}
    \P\left( |C_T^{ij} - C_{\infty}^{ij}| \geq x\right) = \P\left( |e_i^{\star}(C_T - C_{\infty}) e_j| \geq x\right)  \\
     \leq\P\left( \frac{1}{2} \left|(e_i+e_j)^{\star}(C_T - C_{\infty}) (e_i+e_j) + e_i^{\star}(C_T - C_{\infty}) e_i + e_j^{\star}(C_T - C_{\infty}) e_j \right| \geq x\right)\\
     \leq \P\left( \frac{1}{4}\left|(e_i+e_j)^{\star}(C_T - C_{\infty}) (e_i+e_j) \right| \geq \frac{x}{3}\right)
     +\P\left( \left|e_i^{\star}(C_T - C_{\infty}) e_i\right| \geq \frac{x}{3}\right)\\
     +\P\left( \left|e_j^{\star}(C_T - C_{\infty}) e_j\right| \geq \frac{x}{3}\right)
     \leq 6 \exp\left(-\frac{Tx^2}{36\mathcal{K}}\right).
\end{gather}

\end{proof}

\noindent
Now, we state a proposition that provides a bound on the operator norm of $C_{\infty}^{-1}$. This result is a direct consequence of the previous lemma and Weyl's Theorem.

\begin{proposition}\label{propBoundInverse}
     Under Assumptions~\hyperref[assumption:A1]{\((\mathcal{A}_1)\)}-\hyperref[assumption:A3]{\((\mathcal{A}_3)\)} and \hyperref[assumption:C]{\((\mathcal{C})\)}, the inequality 
    \begin{equation}\label{CTdiff}
       \P\left( \nop{C_T^{\supp\supp } - C_\infty^{\supp\supp }} > x \right) \leq  6s^2\exp\left(-\frac{T x^2}{36\mathcal{K} s^2}\right)
    \end{equation}
holds for $x>0$. In particular,  
    \begin{equation}\label{CSSoperator}
       \P\left( \nop{(C_T^{\supp\supp })^{-1}} > \frac{2}{\tau_{\text{min}}} \right) \leq  12 s^2\exp\left(-\frac{T\tau_{\min}^2}{144 \mathcal{K}s^2}\right).
    \end{equation}
    %     \begin{equation}
    %    \P\left( \nop{(C_T^{\supp\supp })^{-1}} > x \right) \leq  12 s^2\exp\left(\frac{-T(\tau_{\min} - 1/x)^2}{36 \mathcal{K}s^2}\right)
    % \end{equation}

\end{proposition}

\begin{proof} We obtain for any $x>0$:

\begin{gather}
    \P\left( \nop{C_T^{\supp\supp } - C_\infty^{\supp\supp }} > x \right) 
    \leq \P \left( \max_{1\leq i \leq s}\sum_{j=1}^s |C_T^{ij} - C_\infty^{ij}|^2>\frac{x^2}{s}\right)\\
    \leq 
    s \max_{1\leq i \leq s} \P \left(\sum_{j=1}^s |C_T^{ij} - C_\infty^{ij}|^2>\frac{x^2}{s}\right)
    \leq s^2 \max_{1\leq i,j \leq s}\P\left(  |C_T^{ij} - C_\infty^{ij}|>\frac{x}{s}\right)\\ 
    \leq 6 s^2 \exp\left(-\frac{Tx^2}{36\mathcal{K}s^2}\right), 
\end{gather}
which establishes inequality \eqref{CTdiff}. 

Now since $C_{T}^{\supp \supp}$ and $C_{\infty}^{\supp \supp}$ are symmetric matrices and $C_{\infty}^{\supp \supp}$ is  positive definite due to the Assumption \hyperref[assumption:C]{\((\mathcal{C})\)}, we can use   Weyl's Theorem (can be consulted in Section 4.3 in \cite{Horn85}) to control the minimal eigenvalue of $C_{T}^{\supp \supp}$ since
\begin{equation}
    l_{\min}(C_{\infty}^{\supp \supp}) - l_{\max}(C_T^{\supp\supp } - C_{\infty}^{\supp \supp}) \leq l_{\min}(C_T^{\supp\supp }) \leq l_{\min}(C_{\infty}^{\supp \supp}) + l_{\max}(C_T^{\supp\supp } - C_{\infty}^{\supp \supp}).
\end{equation}
Consequently, 
\begin{equation}
    l_{\min}(C_T^{\supp\supp }) \geq l_{\min}(C_{\infty}^{\supp \supp}) - \nop{C_T^{\supp\supp } - C_{\infty}^{\supp \supp}} = \tau_{\min} - \nop{C_T^{\supp\supp } - C_{\infty}^{\supp \supp}}. 
\end{equation}
% which leads to 
% \begin{equation}
%     \nop{(C_T^{\supp\supp })^{-1}} = (l_{\min}(C_T^{\supp\supp}))^{-1} \leq \left(\tau_{\min} - \nop{C_T^{\supp\supp } - C_{\infty}^{\supp \supp}}\right)^{-1}.
% \end{equation}
Therefore, we have that 
\begin{gather}
    \P\left( \nop{(C_T^{\supp\supp })^{-1}} > x  \right) \leq 
    \P\left( \nop{(C_T^{\supp\supp })^{-1}} > x ,  \nop{C_T^{\supp\supp } - C_{\infty}^{\supp \supp}} < \tau_\text{min}\right) \\
    + 
    \P\left( \nop{C_T^{\supp\supp } - C_{\infty}^{\supp \supp}} \geq \tau_\text{min}\right)   \leq   
    \P\left(\frac{1}{\tau_{\min} - \nop{C_T^{\supp\supp } - C_{\infty}^{\supp \supp}}} > x,  \nop{C_T^{\supp\supp } - C_{\infty}^{\supp \supp}} < \tau_\text{min} \right)\\
    + 
    \P\left( \nop{C_T^{\supp\supp } - C_{\infty}^{\supp \supp}} \geq \tau_\text{min}\right) 
    \leq 2 \P\left(\nop{C_T^{\supp\supp } - C_{\infty}^{\supp \supp}} >  \tau_{\min} - \frac{1}{x} \right)\\
    \leq 12 s^2\exp\left(-\frac{T(\tau_{\min} - 1/x)^2}{36 \mathcal{K}s^2}\right), 
\end{gather}
where the last inequality follows from \eqref{CTdiff} and holds for $x>1/\tau_{\text{min}}$. Finally, taking $x=2/\tau_{\text{min}}$, we demonstrate inequality \eqref{CSSoperator}.
\end{proof}

The next proposition provides a  bound for the $l_2$-norm of $C_T^{\supp^C \supp}$. 

\begin{proposition}\label{boundL2correlated}  
Under Assumptions~\hyperref[assumption:A1]{\((\mathcal{A}_1)\)}-\hyperref[assumption:A3]{\((\mathcal{A}_3)\)} and \hyperref[assumption:C]{\((\mathcal{C})\)}, the inequality 
    \begin{equation}\label{CTSCSineq}
        \P\left(\|(C_T^{\supp^C \supp})_j\|_2  > x\right) \leq  6 s \exp\left(-\frac{T (x - \mathfrak{L})^2}{36\mathcal{K}s}\right)
    \end{equation}
    holds for $x>\mathfrak{L}$ and $1\leq j \leq p-s$. Moreover, 
    \begin{equation}\label{CTCTineq}
        \P\left(\|(C_T^{\supp^C \supp})_j (C_T^{\supp \supp})^{-1}\|_2  > \frac{2\mathfrak{L}}{\tau_{\text{min}}}+\frac{1}{\sqrt{s}}\right) \leq  6s(2s+1) \exp\left(-\frac{T \tau_{\text{min}}^2}{144\mathcal{K}s^2}\right).
    \end{equation}
\end{proposition}

\begin{proof}
We observe that 
\begin{gather}
    \P\left(\|(C_T^{\supp^C \supp})_j\|_2  > x\right) 
    \leq \P\left( \|(C_T^{\supp^C \supp} - C_\infty^{\supp^C \supp})_j\|_2+\|(C_\infty^{\supp^C \supp})_j\|_2 > x \right)\\ 
    \leq 
    \P\left( \sum_{i=1}^{s} |C_T^{j+s,i} - C_\infty^{j+s,i} |^2  > (x-\mathfrak{L})^2\right) \leq s \max_{1\leq i \leq s}\P\left(|C_T^{j+s,i}-C_\infty^{j+s,i}| > \frac{x-\mathfrak{L}}{\sqrt{s}}\right)\\
    \leq s \max_{1\leq i \leq s}\P\left(|C_T^{j+s,i} - C_\infty^{j+s,i}| > \frac{x - \mathfrak{L}}{\sqrt{s}}\right) \leq 6 s  \exp\left(-\frac{T (x - \mathfrak{L})^2}{36\mathcal{K}s}\right), 
\end{gather}
where the last inequality holds due to Lemma \ref{boundOperatorCovarianceMatrices}. To prove the second inequality, we start with 

\begin{gather}
    \P\left(\|(C_T^{\supp^C \supp})_j (C_T^{\supp \supp})^{-1}\|_2  > \frac{2\mathfrak{L}}{\tau_{\text{min}}}+\frac{1}{\sqrt{s}}\right) \leq
    \P\left(\|(C_T^{\supp^C \supp})_j \|_2  > \mathfrak{L} + \frac{\tau_\text{min}}{2\sqrt{s}}\right) \\
    + 
    \P\left(\nop{(C_T^{\supp \supp})^{-1}}  >\frac{2}{\tau_{\text{min}}}\right), 
\end{gather}
and \eqref{CTCTineq} follows immediately from \eqref{CSSoperator} and \eqref{CTSCSineq}.  
% \begin{proposition}\label{boundL2correlated}

% Under Assumptions ($\mathcal{A}_1$)-($\mathcal{A}_3$) and ($\mathcal{C}$) the inequality 
%     \begin{equation}
%         \P\left(\|(C_T^{\supp^C \supp})_j\|_2  > x\right) \leq  6 s \exp\left(\frac{-T}{36\mathcal{K}}\left( \frac{x}{\sqrt{s}}- \mathfrak{M} \right)^2\right)
%     \end{equation}
%     holds for $x>\mathfrak{M}$. 
% \end{proposition}
% \begin{proof}
% We observe that 
% \begin{gather}
%     \P\left(\|(C_T^{\supp^C \supp})_j\|_2  > x\right) = 
%     \P\left( \sum_{i=1}^s |C_T^{j+s,i} |^2  > x^2\right) \leq s \max_{1\leq i \leq s}\P\left(|C_T^{j+s,i}| > \frac{x}{\sqrt{s}}\right)\\
%     \leq s \max_{1\leq i \leq s}\P\left(|C_T^{j+s,i} - C_\infty^{j+s,i}| + \mathfrak{M}> \frac{x}{\sqrt{s}}\right) \leq 6 s \exp\left(\frac{-T}{36\mathcal{K}}\left( \frac{x}{\sqrt{s}}- \mathfrak{M} \right)^2\right), 
% \end{gather}
% where the last inequality holds due to Lemma \ref{boundOperatorCovarianceMatrices}.
% As in preceding proofs, we observe that for all $i,j = 1,\dots,p$
% \begin{equation}
%     \left\{ |(C_T)_{ji} - (C_{\infty})_{ji}| > x \right\}  \subseteq \left\{ \nop{C_T - C_{\infty}}  > x\right\}.
% \end{equation}
% Therefore, by combining Lemma  \ref{boundOperatorCovarianceMatrices} and bound \eqref{boundEphi} for $|(C_{\infty})_{ji}|$, we obtain that 
% \begin{align}
%     &\P\left(\|(C_T^{\supp^C \supp})_j\|_2  > x\right)  \leq \P\left(|(C_T)_{ji} - (C_{\infty})_{ji}|  > x /\sqrt{s} -dR \right) \leq 2\exp\left(\frac{-T(x/\sqrt{s} -dR)^2}{4K\nop{\mathcal{Q}}^2}\right)
% \end{align}
% which is the desired result.
\end{proof}

\subsection{Bound for  the martingale term}
% Another important part of our reasoning consists in deriving a bound for the martingale part $\epsilon_T$. We present the following result.
Another crucial step in our analysis involves deriving a bound for the martingale component $\epsilon_T$. The result is stated as follows:

\begin{proposition}\label{propMTGbernstein}
% Under Assumptions ($\mathcal{A}_1$)-($\mathcal{A}_3$) and ($\mathcal{C}$) the inequality 
%     \begin{equation}
%         \P\left(|\epsilon_T^j|> x \right) 
%     \leq  2 \exp\left(\frac{-Tx^2}{2y}\right) +  6\exp\left(\frac{-T(y -\mathfrak{M})^2}{36\mathcal{K}}\right).
%     \end{equation}
% holds for all $x,y>0$. 

Under Assumptions~\hyperref[assumption:A1]{\((\mathcal{A}_1)\)}-\hyperref[assumption:A3]{\((\mathcal{A}_3)\)} and \hyperref[assumption:C]{\((\mathcal{C})\)}, the inequalities 
    \begin{equation}\label{epsilonineq}
        \P\left(|\epsilon_T^j|> x \right) 
    \leq  2 \exp\left(-\frac{Tx^2}{2(\sqrt{\mathcal{K}}+\mathfrak{M})}\right) +  6\exp\left(-\frac{T}{36}\right).
    \end{equation}
and
\begin{equation}\label{epsilonSineq}
        \P\left(\|\epsilon_{T,\supp}\|_2> x \right) 
    \leq  2 s\exp\left(-\frac{Tx^2}{2s(\sqrt{\mathcal{K}}+\mathfrak{M})}\right) +  6s\exp\left(-\frac{T}{36}\right).
    \end{equation}
hold for all $x>0$. 
\end{proposition}
\begin{proof}
We observe that the quadratic variation of the martingale $\epsilon_T^j$ can be expressed as 
\begin{equation}
    \langle \epsilon^j\rangle_T = \frac{1}{T^2}\int_0^T\|\phi_j(X_t)\|_2^2dt = \frac{1}{T}C_{T}^{jj}.
\end{equation}
Then, combining Bernstein’s inequality for continuous martingales with inequality \eqref{ijConDif} from Lemma \ref{boundOperatorCovarianceMatrices},  we obtain that for all $x,y>0$,
\begin{gather}
    \P\left(|\epsilon_T^j|> x \right) \leq \P\left( |\epsilon_T^j|> x , \langle M^j\rangle_T  <  \frac{y}{T} \right) + \P\left( \langle M^j\rangle_T  \geq  \frac{y}{T} \right)\\
    \leq 2 \exp\left(-\frac{Tx^2}{2y}\right) + \P\left( C_T^{jj}  \geq  y \right) 
    \leq 2 \exp\left(-\frac{Tx^2}{2y}\right) + \P\left( |C_T^{jj}-C_\infty^{jj}|\geq  y - \mathfrak{M}\right)\\
    \leq 2 \exp\left(-\frac{Tx^2}{2y}\right) +  6\exp\left(-\frac{T(y -\mathfrak{M})^2}{36\mathcal{K}}\right).    
\end{gather}
Taking $y = \sqrt{\mathcal{K}}+\mathfrak{M}$ finishes the proof of \eqref{epsilonineq}. Finally, we can notice that 

\begin{gather}
   \P\left(\|\epsilon_{T,\supp}\|_2> x \right) \leq s \max_{1 \leq j \leq s} \P\left(|\epsilon_T^j|> \frac{x}{\sqrt{s}} \right)\leq
   2 s\exp\left(-\frac{Tx^2}{2s(\sqrt{\mathcal{K}}+\mathfrak{M})}\right) +  6s\exp\left(-\frac{T}{36}\right).
\end{gather}
\end{proof}

\subsection{Bounds on probabilities of key events} 
In this subsection, we present the lemmas that demonstrate the conditions, which the model parameters must satisfy to ensure sign consistency for the adaptive Lasso estimator.
\begin{lemma}\label{LemmaB1}
   Under Assumptions~\hyperref[assumption:A1]{\((\mathcal{A}_1)\)}-\hyperref[assumption:A3]{\((\mathcal{A}_3)\)} and \hyperref[assumption:C]{\((\mathcal{C})\)}, it holds that
   $\P(B_1)\to 0$  if 
\begin{equation}\label{lemmaB1condition}
 \left( \frac{\sqrt{\mathcal{K}}+\mathfrak{M}}{(\theta_{0,\supp}^\text{min})^2}+ \frac{\tau_{\text{min}}^2}{s} + \mathcal{K} s \right) \frac{s\log{s}}{\tau_{\text{min}}^2 T} \to 0.
\end{equation}

\end{lemma}
\begin{proof}
By definition, 
\begin{equation}
    B_1 = \left\{\exists j \in \supp :  \left|
    \left((C_T^{\supp\supp })^{-1} \epsilon_{T,\supp}\right)^j\right| \geq \frac{|\theta_0^j|}{2}\right\}.
\end{equation}
Consequently, applying inequalities \eqref{CSSoperator} and \eqref{epsilonSineq}, we obtain that 
\begin{gather}
    \P(B_1) \leq s \max_{j \in \supp} \P\left(\left|
    \left((C_T^{\supp\supp })^{-1} \epsilon_{T,\supp}\right)^j\right| \geq \frac{|\theta_0^j|}{2} \right) 
    \leq s  \P\left(\nop{(C_T^{\supp\supp})^{-1}} \|\epsilon_{T,\supp}\|_2\geq \frac{\theta_{0,\supp}^\text{min}}{2} \right) \nonumber \\ \nonumber
    \leq s \P\left(\nop{(C_T^{\supp\supp})^{-1}} \|\epsilon_{T,\supp}\|_2\geq \frac{\theta_{0,\supp}^\text{min}}{2}, \nop{(C_T^{\supp\supp})^{-1}}  \leq \frac{2}{\tau_{\text{min}}}  \right) 
    + s \P\left( \nop{(C_T^{\supp\supp})^{-1}}  > \frac{2}{\tau_{\text{min}}}  \right)\\  
    \leq s \P\left(\|\epsilon_{T,\supp}\|_2\geq \frac{\theta_{0,\supp}^\text{min}\tau_{\text{min}}}{4}  \right) 
    + s \P\left( \nop{(C_T^{\supp\supp})^{-1}}  > \frac{2}{\tau_{\text{min}}}  \right) \\
    \leq 2 s^2\exp\left(-\frac{T(\theta_{0,\supp}^\text{min})^2\tau_{\text{min}}^2}{32s(\sqrt{\mathcal{K}}+\mathfrak{M})}\right) +  6s^2\exp\left(-\frac{T}{36}\right)
    +12 s^3\exp\left(-\frac{T\tau_{\min}^2}{144 \mathcal{K}s^2}\right), 
\end{gather}
where all three terms converge to zero under the condition \eqref{lemmaB1condition}. 
\end{proof}

Let recall that $u = (|\eta_0^j|^{-1}\sign(\theta_0), j \in \supp)^{\star}$ and $\widetilde{u} = (w^j\sign(\theta_0^j), j \in \supp)^{\star}$, where $w^j = |\widetilde{\theta}^j|^{-1}$  is derived from the pre-estimator $\widetilde{\theta}$. The following lemma establishes a probabilistic bound on the Euclidean norm of $\widetilde{u}$. 

\begin{lemma}\label{uhat}
Under Assumption \hyperref[assumption:B1]{\((\mathcal{B}_1)\)}, the following equalities hold 
    \begin{equation}\label{uasymp}
        \|\widetilde{u}\|_2 = (1 + o_{\P}(1)) \mathcal{M}_{1,T}\quad \text{and}\quad \max_{j\notin \supp} \| |\widetilde{\theta^j}|\widetilde{u} - |\eta_0^j| u  \|_2 = o_\P(1). 
    \end{equation}
\end{lemma}
\begin{proof} Since $ \mathcal{M}_{1,T} = o(r_T)$ we can notice that 

\begin{equation}
    \max_{j \in \supp }\left| \frac{|\widetilde{\theta}^j|}{|\eta_0^j|} - 1\right|\leq  O_\P\left(\frac{\mathcal{M}_{1,T}}{r_T}\right)=o_\P(1),    
\end{equation}
which proves the first equality. For the second part of \eqref{uasymp} we have 
\begin{equation}
    \max_{j\notin\supp}\| (|\widetilde{\theta}^j| - |\eta_0^j|)\widetilde{u} \|_2 = O_\P\left( \frac{\mathcal{M}_{1,T}}{r_T} \right)=o_\P(1)
\end{equation}
and 
\begin{gather}
    \max_{j\notin\supp}\| |\eta_0^j| \widetilde{u} - |\eta_0^j| u \|_2^2 \leq \mathcal{M}_{2,T}^2 \sum_{j\in\supp}\left| \frac{|\widetilde{\theta}^j| - |\eta_0^j|}{|\widetilde{\theta}^j\eta_0^j|} \right|^2 =O_\P\left( \frac{\mathcal{M}_{1,T}^2}{r_T^2}\right)=o_\P(1). 
\end{gather}

\end{proof}

\begin{lemma}\label{LemmaB2}
 Under Assumptions~\hyperref[assumption:A1]{\((\mathcal{A}_1)\)}-\hyperref[assumption:A3]{\((\mathcal{A}_3)\)}, \hyperref[assumption:B1]{\((\mathcal{B}_1)\)} and \hyperref[assumption:C]{\((\mathcal{C})\)}, it holds that $\P(B_2)\to 0$  if 
\begin{equation}\label{lemmaB2condition}
 \mathcal{K} \frac{s^2\log{s}}{\tau_{\text{min}}^2 T} + \frac{\lambda\mathcal{M}_{1,T}}{\theta_{0,\supp}^\text{min}\tau_{\text{min}}}\to 0.
\end{equation}
\end{lemma}
\begin{proof}

By definition, 
\begin{equation}
    B_2 = \left\{\exists j \in \supp :  \left|
    \left((C_T^{\supp\supp })^{-1} \widetilde{u}\right)^j\right| \geq \frac{|\theta_0^j|}{2 \lambda}\right\}.
\end{equation}
Consequently, applying inequality \eqref{CSSoperator} we can bound 
\begin{gather}\label{eqq}
    \P(B_2) \leq s \max_{j \in \supp} \P\left(\left|
    \left((C_T^{\supp\supp })^{-1} \widetilde{u}\right)^j\right| \geq \frac{|\theta_0^j|}{2\lambda} \right) 
    \leq s  \P\left(\nop{(C_T^{\supp\supp})^{-1}} \|\widetilde{u}\|_2\geq \frac{\theta_{0,\supp}^\text{min}}{2\lambda} \right) \nonumber \\ \nonumber
    \leq s \P\left(\nop{(C_T^{\supp\supp})^{-1}} \|\widetilde{u}\|_2\geq \frac{\theta_{0,\supp}^\text{min}}{2\lambda}, \nop{(C_T^{\supp\supp})^{-1}}  \leq \frac{2}{\tau_{\text{min}}}  \right) 
    + s \P\left( \nop{(C_T^{\supp\supp})^{-1}}  > \frac{2}{\tau_{\text{min}}}  \right)\\  
    \leq s \P\left(\frac{\|\widetilde{u}\|_2}{\mathcal{M}_{1,T}}\geq \frac{\theta_{0,\supp}^\text{min}\tau_{\text{min}}}{4\lambda \mathcal{M}_{1,T}}  \right) 
    + 12 s^3\exp\left(-\frac{T\tau_{\min}^2}{144 \mathcal{K}s^2}\right)
\end{gather}
and according to the first equality in \eqref{uasymp}, both of these terms converge to zero under  condition \eqref{lemmaB2condition}. 

\end{proof}

\begin{lemma}\label{LemmaB3}
 Under Assumptions~\hyperref[assumption:A1]{\((\mathcal{A}_1)\)}-\hyperref[assumption:A3]{\((\mathcal{A}_3)\)}, \hyperref[assumption:B1]{\((\mathcal{B}_1)\)} and \hyperref[assumption:C]{\((\mathcal{C})\)}, it holds that $\P(B_3)\to 0$ and $\P(B_4)\to 0$  if 
\begin{equation}\label{lemmaB3condition}
 \left(  \frac{ (\sqrt{\mathcal{K}}+\mathfrak{M})(\mathcal{M}_{2,T} + \frac{1}{r_T})^2(\frac{\tau_{\text{min}}}{\sqrt{s}}+2\mathfrak{L})^2}{\lambda^2} + \frac{\tau_\text{min}^2}{s}+\mathcal{K}s\right)\frac{s\log{p}}{\tau_{\text{min}}^2 T}\to 0.
\end{equation}
\end{lemma}

\begin{proof}
Due to  Assumption \hyperref[assumption:B1]{\((\mathcal{B}_1)\)}, we have that for all $j \notin \supp$
\begin{equation}\label{boundInverseTildeTheta}
    \frac{1}{w^j} = |\widetilde{\theta^j}| \leq \mathcal{M}_{2,T} + O_\P \left(\frac{1}{r_T}\right)
\end{equation}
and since by definition 
\begin{equation}
B_3 = \left\{\exists j \notin \supp :  \left| \left( C_T^{\supp^C \supp} (C_T^{\supp\supp })^{-1} \epsilon_{T, \supp} \right)^{j-s}\right| \geq \frac{1-\kappa - \varepsilon}{2}\lambda w^j\right\},
\end{equation}
and 
\begin{equation}
B_4 = \left\{\exists j \notin \supp :  \left| \epsilon_T^j\right| \geq \frac{1-\kappa - \varepsilon}{2}\lambda w^j\right\},
\end{equation}
for large constant $C>0$, we have that 
\begin{gather}
    \P(B_3) + \P(B_4) = \P \left( \exists j \notin \supp :  \left| \left( C_T^{\supp^C \supp} (C_T^{\supp\supp })^{-1} \epsilon_{T,\supp} \right)^{j-s}\right| \geq \frac{(1-\kappa - \varepsilon)\lambda}{2C\left(\mathcal{M}_{2,T} 
     + \frac{1}{r_T}\right)}\right) \\
     +\P\left( \exists j \notin \supp :  \left| \epsilon_T^j\right| \geq \frac{(1-\kappa - \varepsilon)\lambda}{2C\left(\mathcal{M}_{2,T} 
     + \frac{1}{r_T}\right)} \right)+ o_{\P}(1). 
\end{gather}
First, we observe that by using inequalities \eqref{epsilonSineq} and \eqref{CTCTineq}, we deduce that 
\begin{gather}
    \P \left( \exists j \notin \supp :  \left| \left( C_T^{\supp^C \supp}(C_T^{\supp\supp })^{-1} \epsilon_{T,\supp} \right)^{j-s}\right| \geq \frac{(1-\kappa - \varepsilon)\lambda}{2C\left(\mathcal{M}_{2,T} + \frac{1}{r_T}\right)}\right)\\
    \leq 
    (p-s)\max_{j\notin\supp} \P \left( \left| \left( C_T^{\supp^C \supp} (C_T^{\supp\supp })^{-1} \epsilon_{T,\supp} \right)^{j-s}\right| \geq \frac{(1-\kappa - \varepsilon)\lambda}{2C\left(\mathcal{M}_{2,T} + \frac{1}{r_T}\right)}\right)\\
    \leq (p-s) \P \left( \| \epsilon_{T,\supp} \|_2 \geq \frac{(1-\kappa - \varepsilon)\lambda}{2C\left(\mathcal{M}_{2,T} + \frac{1}{r_T}\right)\left(\frac{1}{\sqrt{s}}+\frac{2\mathfrak{L}}{\tau_{\text{min}}}\right)}\right)\\
    +(p-s)\max_{j\notin\supp}\P\left(\|(C_T^{\supp^C \supp})_{j-s} (C_T^{\supp \supp})^{-1}\|_2  > \frac{1}{\sqrt{s}}+\frac{2\mathfrak{L}}{\tau_{\text{min}}}\right)\\
    \leq 2 s (p-s)\exp\left(-\frac{T(1-\kappa - \varepsilon)^2\lambda^2}{8C^2s\left(\sqrt{\mathcal{K}}+\mathfrak{M}\right)\left(\mathcal{M}_{2,T} + \frac{1}{r_T}\right)^2\left(\frac{1}{\sqrt{s}}+\frac{2\mathfrak{L}}{\tau_{\text{min}}}\right)^2}\right) \\
    +  6s (p-s)\exp\left(-\frac{T}{36}\right)+6s(2s+1)(p-s) \exp\left(-\frac{T \tau_{\text{min}}^2}{144\mathcal{K}s^2}\right), 
\end{gather}
and all three terms converge to zero under the condition \eqref{lemmaB3condition}. Similarly, we see that 

\begin{gather}
    \P\left( \exists j \notin \supp :  \left| \epsilon_T^j\right| \geq \frac{(1-\kappa - \varepsilon)\lambda}{2C\left(\mathcal{M}_{2,T} 
     + \frac{1}{r_T}\right)} \right)\leq 
     (p-s)\max_{j\notin\supp} \P\left( \left| \epsilon_T^j\right| \geq \frac{(1-\kappa - \varepsilon)\lambda}{2C\left(\mathcal{M}_{2,T} 
     + \frac{1}{r_T}\right)} \right)\\
     \leq 2 (p-s)\exp\left(-\frac{T(1-\kappa - \varepsilon)^2\lambda^2}{8C^2 s \left(\sqrt{\mathcal{K}}+\mathfrak{M}\right)\left(\mathcal{M}_{2,T} + \frac{1}{r_T}\right)^2\left(\frac{1}{\sqrt{s}}+\frac{2\mathfrak{L}}{\tau_{\text{min}}}\right)^2}\right) 
     +  6 (p-s)\exp\left(-\frac{T}{36}\right),
\end{gather}
which also converge to zero under the condition \eqref{lemmaB3condition}.

\end{proof}

% \subsection{Set $B_4$}
% \begin{lemma}\label{LemmaB4}
%     Let us consider \cref{A1}- \cref{A6}. $\P(B_4) \rightarrow 0.$
% \end{lemma}
% \begin{proof}
% We recall that $B_4$ is defined as 
% \begin{equation}
%     B_4 = \left\{\exists j \notin \supp \Big|  \left| \epsilon_T^j\right| \geq \frac{\lambda w^j}{5}\right\}.
% \end{equation}
% We observe directly that we can bound $B_4$ from
% \begin{align}
%     \P(B_4) \leq (p-s)\max_{j \notin \supp} \P\left( \left| \epsilon_T^j\right| \geq \frac{\lambda w^j}{5} \right).
% \end{align}
% Using Proposition \ref{propMTGbernstein}, where we choose $y=\tau_{\min}/2 + dR$,  we obtain that 
%  \begin{equation}
%       \max_{j \notin \supp} \P\left(|\epsilon_T^j|> \frac{\lambda w^j}{5} \right) 
%     \leq 2\exp\left(\frac{-T \lambda^2 {w^j}^2}{50(\tau_{\min}/2 + dR)}\right) + 4\exp\left(\frac{-T{\tau^2_{\min}}}{16K\nop{\mathcal{Q}}^2}\right).
% \end{equation}
% Finally, combining Condition \eqref{C2}, the consistency of the Lasso estimator $\eqref{lassoConsistency}$ and Condition \eqref{C5}, we prove that $\P(B_4) \rightarrow 0$.
% \end{proof}

\begin{lemma}\label{LemmaB5}
 Under Assumptions~\hyperref[assumption:A1]{\((\mathcal{A}_1)\)}-\hyperref[assumption:A3]{\((\mathcal{A}_3)\)}, \hyperref[assumption:B1]{\((\mathcal{B}_1)\)}-\hyperref[assumption:B2]{\((\mathcal{B}_2)\)} and \hyperref[assumption:C]{\((\mathcal{C})\)}, it holds that $\P(B_5)\to 0$  if condition \eqref{lemmaB3condition} holds and 
\begin{equation}\label{lemmaB5condition}
\left( \frac{1}{s} + \frac{\mathfrak{L^2}}{\tau_\text{min}^2} \right) \frac{\mathcal{K}s^2\log{p}}{\tau_{\text{min}}^2 T}\mathcal{M}_{1,T}^2\to 0.
\end{equation}
\end{lemma}

\begin{proof} By definition,

\begin{equation}
    B_5 = \left\{\exists j \notin \supp :  \left| \left( C_T^{\supp^C \supp} (C_T^{\supp\supp })^{-1}\widetilde{u}\right)^{j-s} \right| \geq (\kappa + \varepsilon) w^j\right\}.
\end{equation}
According to the adaptive irrepresentable condition  \hyperref[assumption:B2]{\((\mathcal{B}_2)\)}, and recalling that $w^j = |\widetilde{\theta}^j|$, we can bound the probability as
\begin{gather}
    \P(B_5)\leq \P\left(\exists j \notin \supp :  \left| \left( C_T^{\supp^C \supp} (C_T^{\supp\supp })^{-1}|\widetilde{\theta}^j|\widetilde{u} -  C_\infty^{\supp^C \supp} (C_\infty^{\supp\supp })^{-1}|\eta_0^j|u\right)^{j-s} \right| \geq \varepsilon\right)\\
    \leq (p-s) \max_{j\notin\supp}\P\left( \left| \left( C_T^{\supp^C \supp} (C_T^{\supp\supp })^{-1}\left( |\widetilde{\theta}^j|\widetilde{u} - |\eta_0^j|u\right)\right)^{j-s} \right| \geq \varepsilon/2 \right)\\
    +(p-s) \max_{j\notin\supp}\P\left( \left| \left( \left(C_T^{\supp^C \supp} (C_T^{\supp\supp })^{-1}-  C_\infty^{\supp^C \supp} (C_\infty^{\supp\supp })^{-1}\right)|\eta_0^j|u\right)^{j-s} \right| \geq \varepsilon /2\right).
\end{gather}
First,  we  show that for any $j\notin\supp$,
\begin{gather}
    (p-s) \max_{j\notin\supp}\P\left( \left| \left( C_T^{\supp^C \supp} (C_T^{\supp\supp })^{-1}\left( |\widetilde{\theta}^j|\widetilde{u} - |\eta_0^j|u\right)\right)^{j-s} \right| \geq \varepsilon/2 \right)\\
    \leq (p-s) \P\left( \max_{j\notin\supp}\| |\widetilde{\theta}^j|\widetilde{u} - |\eta_0^j|u\|_2 \left(\frac{1}{\sqrt{s}}+\frac{2\mathfrak{L}}{\tau_{\text{min}}}\right) \geq\varepsilon/2\right) \\
    +(p-s)\P\left(\|(C_T^{\supp^C \supp})_{j-s} (C_T^{\supp \supp})^{-1}\|_2  > \frac{1}{\sqrt{s}}+\frac{2\mathfrak{L}}{\tau_{\text{min}}}\right), 
\end{gather}
where both terms converge to zero due to \eqref{uasymp} and \eqref{lemmaB3condition}.  Finally, since $|\eta_0^j|\| u \|_2 \leq \mathcal{M}_{1,T}$, we deduce that 

\begin{gather}
    (p-s) \max_{j\notin\supp}\P\left( \left| \left( \left(C_T^{\supp^C \supp} (C_T^{\supp\supp })^{-1}-  C_\infty^{\supp^C \supp} (C_\infty^{\supp\supp })^{-1}\right)|\eta_0^j|u\right)^{j-s} \right| \geq \varepsilon /2\right)\nonumber\\
    \leq 
    (p-s) \max_{j\notin\supp}
    \P\left( \left\| \left( C_T^{\supp^C \supp} (C_T^{\supp\supp })^{-1}-  C_\infty^{\supp^C \supp} (C_\infty^{\supp\supp })^{-1}\right)_{j-s} \right\|_2 \geq \frac{\varepsilon}{2 \mathcal{M}_{1,T}}\right) \nonumber\\
    %\leq 24 s^2 (p-s)  \exp\left(\frac{-T \varepsilon^2 \tau_\text{min}^2}{576\mathcal{K}s \mathcal{M}_{1,T}^2}\right)
    \leq (p-s) \max_{j\notin\supp}
    \P\left( \left\| \left( C_T^{\supp^C \supp} (C_T^{\supp\supp })^{-1}-  C_\infty^{\supp^C \supp} (C_T^{\supp\supp })^{-1}\right)_{j-s} \right\|_2 \geq \frac{\varepsilon}{4 \mathcal{M}_{1,T}}\right) \label{prob5first}\\
    + (p-s) \max_{j\notin\supp}
    \P\left( \left\| \left( C_\infty^{\supp^C \supp} (C_T^{\supp\supp })^{-1}-  C_\infty^{\supp^C \supp} (C_\infty^{\supp\supp })^{-1}\right)_{j-s} \right\|_2 \geq \frac{\varepsilon}{4 \mathcal{M}_{1,T}}\right) \label{prob5second}. 
\end{gather}
Now, we analyse these probabilities separately. Starting with \eqref{prob5first}, using inequalities  \eqref{CTSCSineq} and \eqref{CSSoperator}, we have that 
\begin{gather}
    (p-s) \max_{j\notin\supp}
    \P\left( \left\| \left( C_T^{\supp^C \supp} (C_T^{\supp\supp })^{-1}-  C_\infty^{\supp^C \supp} (C_T^{\supp\supp })^{-1}\right)_{j-s} \right\|_2 \geq \frac{\varepsilon}{4 \mathcal{M}_{1,T}}\right)\\
    \leq (p-s) \max_{j\notin\supp}
    \P\left( \left\| \left( C_T^{\supp^C \supp} -  C_\infty^{\supp^C \supp} \right)_{j-s} \right\|_2 \geq \frac{\varepsilon\tau_\text{min}}{8 \mathcal{M}_{1,T}}\right)+(p-s)\P\left( \nop{(C_T^{\supp\supp })^{-1}}>\frac{2}{\tau_\text{min}} \right)\\
    \leq
    6 s (p-s)\exp\left(-\frac{T \varepsilon^2 \tau_\text{min}^2}{2304\mathcal{K}s \mathcal{M}_{1,T}^2}\right)+12 s^2(p-s)\exp\left(-\frac{T\tau_{\min}^2}{144 \mathcal{K}s^2}\right),
\end{gather}
and both terms converge to zero if  \eqref{lemmaB3condition} and \eqref{lemmaB5condition} hold. Finally, we can analogously bound \eqref{prob5second} by using \eqref{CTdiff} and \eqref{CSSoperator} as
\begin{gather}
    (p-s) \max_{j\notin\supp}
    \P\left( \left\| \left( C_\infty^{\supp^C \supp} (C_T^{\supp\supp })^{-1}-  C_\infty^{\supp^C \supp} (C_\infty^{\supp\supp })^{-1}\right)_{j-s} \right\|_2 \geq \frac{\varepsilon}{4 \mathcal{M}_{1,T}}\right)\\
    \leq
    (p-s) \max_{j\notin\supp}
    \P\left( \left\|  \left( C_\infty^{\supp^C \supp}\right)_{j-s}\right\|_2 \nop{(C_T^{\supp\supp })^{-1}} \nop{ C_T^{\supp\supp } - C_\infty^{\supp\supp }}\nop{(C_\infty^{\supp\supp })^{-1}} \geq \frac{\varepsilon}{4 \mathcal{M}_{1,T}}\right)\\
    \leq (p-s) 
    \P\left(  \nop{ C_T^{\supp\supp } - C_\infty^{\supp\supp }} \geq \frac{\varepsilon \tau_\text{min}^2}{8\mathfrak{L} \mathcal{M}_{1,T}}\right)+(p-s)\P\left( \nop{(C_T^{\supp\supp })^{-1}}>\frac{2}{\tau_\text{min}} \right)\\
    \leq
    6s^2(p-s)\exp\left( -\frac{T \varepsilon^2 \tau_\text{min}^4}{2304\mathcal{K}\mathfrak{L}^2s^2 \mathcal{M}_{1,T}^2} \right) 
    +12 s^2(p-s)\exp\left(-\frac{T\tau_{\min}^2}{144 \mathcal{K}s^2}\right), 
\end{gather}
and again, both terms converge to 0 since \eqref{lemmaB3condition} and \eqref{lemmaB5condition} hold.

\end{proof}

\subsubsection*{Acknowledgment} 

Francisco Pina's research is funded by the PRIDE Grant “MATHCODA: Mathematical Tools for Complex Data Structures”. Mark Podolskij expresses his gratitude for the financial support provided by the ERC Consolidator Grant 815703, titled “STAMFORD: Statistical Methods for High-Dimensional Diffusions”.

\bibliography{main}

\end{document}